\documentclass{article}

\usepackage[english]{babel}
\usepackage{soul}

\usepackage[letterpaper,top=2cm,bottom=2cm,left=3cm,right=3cm,marginparwidth=1.75cm]{geometry}

\usepackage{amsmath}
\usepackage{amsthm}
\usepackage{graphicx}
\usepackage[colorlinks=true, allcolors=blue]{hyperref}
\usepackage{algpseudocode}
\usepackage{enumerate,bm}
\usepackage{graphicx}
\usepackage{rotating}   
\usepackage[outdir=./]{epstopdf}
\usepackage{xcolor}
\usepackage{soul}
\usepackage{environ}
\pagecolor{white}
\usepackage[linesnumbered,ruled]{algorithm2e}
\usepackage{multirow}
\usepackage{amssymb,amsmath,mathtools}
\usepackage{subcaption}

\DeclareMathOperator*{\argmin}{arg\,min}                   

\renewcommand{\phi}{\mathbf{\varphi}}

\newcommand{\bfT}{\mathbf{T}}

\newcommand{\bfA}{\mathbf{A}}

\newcommand{\bfU}{\mathbf{U}}
\newcommand{\bfP}{\mathbf{P}}
\newcommand{\bfR}{\mathbf{R}}

\newcommand{\bfb}{\mathbf{b}}

\newcommand{\bfH}{\mathbf{H}}

\newcommand{\bfX}{\mathbf{X}}
\newcommand{\bfx}{\mathbf{x}}

\newcommand{\bfe}{\mathbf{e}}

\newcommand{\bfy}{\mathbf{y}}

\newcommand{\bfM}{\mathbf{M}}

\newcommand{\bfL}{\mathbf{L}}

\newcommand{\bfW}{\mathbf{W}}

\newcommand{\bfh}{\mathbf{h}}
\newcommand{\bfr}{\mathbf{r}}

\newcommand{\bfV}{\mathbf{V}}

\newcommand{\bfQ}{\mathbf{Q}}

\newcommand{\bfZ}{\mathbf{Z}}

\newcommand{\revision}[1]{\textcolor{black}{#1}}
\newcommand{\revmsl}[1]{\textcolor{black}{#1}}
\renewcommand\st[1]{}

\newtheorem{theorem}{Theorem}[section]
\newtheorem{proposition}[theorem]{Proposition}
\newtheorem{corollary}[theorem]{Corollary}

\title{Iterative Refinement and Flexible Iteratively Reweighed Solvers for Linear Inverse Problems with Sparse Solutions}

\makeatletter
\renewcommand\@date{{%
  \vspace{-\baselineskip}%
  \large\centering
  \begin{tabular}{@{}c@{}}
    Lucas Onisk\textsuperscript{1} \\
    \normalsize lonisk@emory.edu
  \end{tabular}%
  \quad and\quad
  \begin{tabular}{@{}c@{}}
    Malena Sabaté Landman\textsuperscript{2} \\
    \normalsize Malena.SabateLandman@maths.ox.ac.uk
  \end{tabular}
  \bigskip

  \textsuperscript{1}\normalsize Department of Mathematics, Emory University, Atlanta, GA, USA\par
  \textsuperscript{2}\normalsize Mathematical Institute, University of Oxford, Oxford, UK

}}
\makeatother

\begin{document}
\maketitle

\begin{abstract}
This paper presents a new algorithmic framework for computing sparse solutions to large-scale linear discrete ill-posed problems. The approach is motivated by recent perspectives on iteratively reweighted norm schemes, viewed through the lens of iterative refinement. This framework leverages the efficiency and fast convergence of flexible Krylov methods while achieving higher accuracy through suitable restarts. Additionally, we demonstrate that the proposed methods outperform other flexible Krylov approaches in memory-limited scenarios. Relevant convergence theory is discussed, and the performance of the proposed algorithms is illustrated through a range of numerical examples, including image deblurring and computed tomography.
\end{abstract}

\section{Introduction}\label{sec:Intro}

In this paper we consider discrete linear inverse problems of the form:
\begin{equation}\label{eq:original}
    \mathbf{A}\mathbf{x_{\textbf{true}}}+ \boldsymbol{\epsilon}=\mathbf{b},
\end{equation}
where $\mathbf{A}\in\mathbb{R}^{m\times n}$ is a discretization of the forward model, and the right-hand side $\mathbf{b}$ is corrupted by white Gaussian noise $\boldsymbol{\epsilon}$. In many applications, problem \eqref{eq:original} is ill-posed in the sense that the singular values of the matrix $\mathbf{A}$ decay \revision{without significant gap} and cluster at zero, and the properties of the forward model make the solution very sensitive to noise in the measurements. In this case, the least-squares solution of \eqref{eq:original} is not a good approximation of the true solution $\mathbf{x_{\textbf{true}}}$, and we need to resort to the use of regularization. 

Choosing a good regularization method relies on making implicit or explicit assumptions on the solution we want to reconstruct. For example, a common choice of regularization is to consider the following minimization problem,
\begin{equation}\label{eq:l1}
    \min_\mathbf{x} \lVert \mathbf{A} \mathbf{x} - \mathbf{b}\rVert_2^2 + \lambda \lVert \mathbf{L} \mathbf{x}\rVert^p_p,
\end{equation}
which is known to promote sparsity in the reconstructed solutions after a possible linear transformation $\mathbf{L}$ when $ 0 < p \leq 1$. It is also typical to assume that $\mathcal{N}(\bfA) \cap \mathcal{N}{(\bfL)}$ where $\mathcal{N}(\cdot)$ denotes the null space of a matrix, since then the solution of \eqref{eq:l1} is unique for $p=1$. In particular, $\mathbf{L}$ is usually chosen to be a change of basis such as a wavelet transform, see e.g. \cite{belge2000wavelet}, or a discrete approximation of a differential operator: for $p=1$, this leads to a particular approximation of the discrete total variation (TV) regularization. 
\st{Note that, however,} \revision{Note, however, that} considering \eqref{eq:l1} with an $\ell_0$ semi-norm leads to an NP hard optimization problem \cite{Fornasier2011}, and \eqref{eq:l1} is non-convex for $0 < p < 1$, so it is common to only consider $p=1$ in practice, which corresponds to the convex relaxation of \eqref{eq:l1} with an $\ell_0$ semi-norm. 

There exist a vast number of optimization methods that can be used to solve \eqref{eq:l1}. \revmsl{ For example, but not exhaustively, those combining gradient and proximal steps (using acceleration) such as the Fast Iterative Shrinkage-Thresholding Algorithm (FISTA) \cite{beck2009fast}, using split Bregman iterations \cite{goldstein2009split}, augmented Lagrangian methods \cite{AugmentedLagrangian}, or exploiting separable approximations as in the Sparse Reconstruction by Separable Approximation (SpaRSA) \cite{wright2009sparse}.} Other popular approaches\revmsl{, inspired by the numerical linear algebra community,} are based on iteratively reweighted norm (IRN) or iteratively reweighted least-squares (IRLS) schemes, originally \revision{discussed} in \cite{daubechies2010,WohRod}. These are based on constructing a sequence of approximations to a (smoothed) version of the original functional, \revmsl{so that a new sub-problem is} solved at each iteration. \revmsl{Moreover,} these can be used in combination with Krylov subspace solvers\revmsl{, since they show very fast convergence. In this case, however, this leads to a nested loop of iterations, since the solution of each of the sub-problems is approximated using an iterative method.  Here, the outer-problem corresponds to updating the approximation of the functional (by adapting the regularization matrix, $\mathbf{L}$, using the new available approximate solution), and the inner-problem corresponds to finding an approximation of each of the sub-problems using a Krylov method.} 

In this paper, we recast the IRN framework in an iterative refinement scheme, which focuses on optimizing over the update of the solution (rather than the full solution) and correcting the residual at each iteration. \revmsl{Note that this scheme, introduced by Wilkinson in \cite{Wilk48}, was developed to improve the accuracy of the solution to linear systems; a history of IR applied to linear problems may be found in \cite{CarsonHigham18, Higham22}. }However, this is not straightforward in the IRN case and needs to be done carefully to assure the updated systems are equivalent. 

\revmsl{Furthermore, the double looping algorithmic structures of these methods, in what we will term \emph{inner-outer} schemes,  can lead to slow convergence \cite[Chapter 4]{Bjorck}.}
\st{Moreover, these methods are based on inner-outer schemes, which can lead to slow convergence (Bjorck citation).}
For this reason, Krylov subspace variants that  construct adaptive subspaces that can handle the changing weights throughout the iterations, and thus avoiding nested loops, are becoming very popular. These generally consist of two types: generalized Krylov subspaces, e.g. \cite{Lanza2015AGK, lplq, Buccini2024software} to name a few, which are augmented at each iteration using the residual of the normal equations, and flexible Krylov methods \cite{Gazzola2014GAT, JulianneSilvia,Gazzola2019TV, Gazzola2021IRW}, which use iteration-dependent preconditioning. However, these methods can either \revision{prematurely stagnate,} \revmsl{in the sense that the solution updates become arbitrarily small while  the solution is still far from the desired approximation of $\mathbf{x_{\textbf{true}}}$},  or become very memory consuming if the number of iterations is large, \revision{thus necessitating the the ability to restart}. This has been done for generalized Krylov subspaces in \cite{Buccini2023restart}, and including recycling \cite{pasha2023recycling}, but the methodology employed in that paper is not easily applicable to flexible Krylov methods. 
\revision{To allow for flexible methods that employ inherent restarting, we resort to the IRN interpretation as an IR scheme. 
Our method shares similarities to recent work on large linear systems, where the cost of approximating each update step is reduced by
constructing a basis for a (Krylov) subspace of increasing dimension so that each problem in the iterative refinement scheme is projected in a different (nested) subspace \cite{Onisk2022Arnoldi}.} 

In summary, we propose a new \revision{methodology} which combines iteratively reweighted flexible Krylov methods and iterative refinement \revmsl{and present four new specific methods to solve large linear inverse problems with sparse solutions.} These show a comparable convergence to \revmsl{their standard counterparts}, but allow for suitable restarts. This can be convenient to reduce the required memory, and overcome the stagnation that can happen when using iteratively reweighted flexible Krylov methods.

The novelties of this paper are highlighted as follows:
\begin{itemize}
    \item we describe a new perspective on IRN schemes, which is based on iterative refinement,
    \item we propose two new algorithms which combine iteratively reweighted flexible methods and iterative refinement: IR-FGMRES and IR-FLSQR, which allow for seamless restarts. \revision{To the best of our knowledge, this} is the first restarted version of iteratively reweighted flexible methods,
    \item we propose variants of the previous methods which, after each restart, incorporate the most recent approximation to the solution into the search space for the update. These new corrected methods: CIR-FGMRES and CIR-FLSQR, are particularly effective when the restarts happen after a small amount of iterations. 
\end{itemize}

The paper is organized as follows. In Section \ref{Sec-IR-IRN} we give the required background on IRN schemes and flexible Krylov methods. In Section \ref{Sec-IR-IRN} we recast \revision{the IRN framework} \st{IRN frameworks} as a particular case of iterative refinement \revision{and provide necessary background.} \revision{Then, in Section \ref{Sec-new-method}, we propose our algorithms IR-FGMRES and IR-FLSQR as well as their corrected variants CIR-FGMRES and CIR-FLSQR. In Section \ref{sec:theory} we discuss convergence behavior of our proposed methods.}\st{Then, in Section 4, we propose the solvers for linear inverse problems with sparse solutions: IR-FGMRES and IR-FLSQR, as well as corrected variants that incorporate the previous approximation into the solution space after restarts: CIR-FGMRES and CIR-FLSQR.} Finally, some numerical examples are given in Section \ref{sec:numerics} \revision{ that showcase the performance of the new methods. We conclude in Section \ref{sec:conclusion} with some remarks.}

\section{IRN schemes for $\ell_p$ regularization}\label{Sec-IR-IRN}

In this section, we provide some background for IRN and flexible Krylov methods; as well as a brief discussion on existing literature that is used as a comparison in the numerical experiments in Section~\ref{sec:numerics}.

\revision{We use the following notation throughout:} $[\mathbf{A}]_{i,j}$ is the component of the matrix $\mathbf{A}$ in the $i$th row and $j$th column, $\mathbf{a}_k$ is its $k$th column vector, and $[\mathbf{x}]_i$ is $i$-th component of the vector $\mathbf{x}$.

\subsection{Background on IRN}\label{Sec-IRN}

There are several closely related techniques to solve \eqref{eq:l1} which fall into the category of iterative reweighted (IRW) schemes. These optimization methods follow a majorization-minimization approach, and aim to solve variational regularization problems involving $\ell_p$ norms by constructing a sequence of least-squares sub-problems that need to be solved sequentially. In particular, these are quadratic tangent majorants of the original problem. Recall that a quadratic tangent majorant of a functional at a point $\bar{\mathbf{x}}$ is a quadratic upper bound constructed in such a way that the values of both the functional and the gradient of both functionals coincide at $\bar{\mathbf{x}}$.

The fundamental idea behind IRW schemes can be understood by observing that any $\ell_p$ norm can be written as a non-linear weighted $\ell_2$ norm (where $0 < p \leq 2$):
\begin{equation}\label{eq:l1weights}
\Phi(\mathbf{x}) = \frac{1}{p}\left\|\mathbf{L} \mathbf{x}\right\|^p_p = \frac{1}{p} \left\|W(\mathbf{L} \mathbf{x}) \mathbf{L} \mathbf{x}\right\|_2^2   \quad \text{for} \quad \left[W(\mathbf{L}\mathbf{x})\right]_{ii}=\left([\mathbf{L}\mathbf{x}]_{i}\right)^{\frac{p-2}{2}}, 
\end{equation} 
where the regularization matrix \revision{$\mathbf{L}\in \mathbb{R}^{s \times n}$ is given and $W(\mathbf{L}\mathbf{x})\in \mathbb{R}^{s \times s}$} is a diagonal matrix whose elements are often referred to as weights. This expression can be used to construct quadratic tangent majorants of \eqref{eq:l1weights} at any vector $\bar{\mathbf{x}}$
as 
\begin{equation}\label{majorant}
    \Phi^{\bar{\mathbf{x}}}(\mathbf{x})=\frac{1}{2}\lVert W(\mathbf{L}\bar{\mathbf{x}})\mathbf{L}\mathbf{x}\rVert_2^2 + \left(\frac{1}{p}-\frac{1}{2}\right)\lVert W(\mathbf{L}\bar{\mathbf{x}})\mathbf{L}\bar{\mathbf{x}}\rVert_2^2.
\end{equation} 
\revision{From \eqref{majorant} the following can be shown (see \cite{lplq} for further details):
\begin{itemize}
    \item $\Phi^{\bar{\mathbf{x}}}$ is quadratic,
    \item $\Phi^{\bar{\mathbf{x}}}(\mathbf{L}\bar{\mathbf{x}}) = \Phi(\mathbf{L}\bar{\mathbf{x}})$,
    \item the gradients are equal at $\bar{\mathbf{x}}$ (i.e., $\nabla \Phi^{\bar{\mathbf{x}}}(\bar{\mathbf{x}}) = \nabla \Phi(\bar{\mathbf{x}})$),
    \item $\Phi^{\bar{\mathbf{x}}}(\mathbf{x})$ is an upper bound of $\Phi(\mathbf{x})$ for all $\mathbf{x}$.
\end{itemize}}
\st{It can be easily seen that $\Phi^{\bar{\mathbf{x}}}$ is quadratic, that 
$\Phi^{\bar{\mathbf{x}}}(\mathbf{L}\bar{\mathbf{x}}) = \Phi(\mathbf{L}\bar{\mathbf{x}})$, 
that the gradients are equal at $\bar{\mathbf{x}}$ (i.e., $\nabla \Phi^{\bar{\mathbf{x}}}(\bar{\mathbf{x}}) = \nabla \Phi(\bar{\mathbf{x}})$), and that $\Phi^{\bar{\mathbf{x}}}(\mathbf{x})$ is an upper bound of $\Phi(\mathbf{x})$ for all $\mathbf{x}$ (see, e.g., \cite{lplq}).} 

However, in practice, there is a caveat with the functional \eqref{majorant} that is related to the lack of smoothness of the $\ell_p$ norm at the origin when $0 < p \leq 1$. Numerically, the evaluation of \eqref{majorant} at any vector with a 0 valued component will lead to division by zero in the weights in \eqref{eq:l1weights}. 
In this paper, we follow the popular approach of smoothing the functional near zero, i.e., replacing the weights in \eqref{eq:l1weights} by 
\begin{equation}\label{eq:l1weightssmooth}
 [\widetilde{W}(\bfL \mathbf{x})]_{ii}=([\bfL \mathbf{x}]^2_{i}+\tau^2)^{\frac{p-2}{4}}.
\end{equation} 

By doing so, an approximate solution to problem \eqref{eq:l1} can be found iteratively by building a sequence of quadratic tangent majorants of the smoothed functional, and using the solution of each problem to build the next problem in the sequence. This leads to the following recursion:
\begin{equation}\label{recursion}
\mathbf{x}_{k} = \argmin_\mathbf{x} \lVert \mathbf{A} \mathbf{x} - \mathbf{b}\rVert_2^2+ \lambda_k \lVert \mathbf{W}_k \mathbf{L} \mathbf{x}\rVert_2^2, \quad \mathbf{W}_k = \widetilde{W}(\mathbf{L} \mathbf{x}_{k-1}),
\end{equation}
where we have dropped the terms that do not depend on $\mathbf{x}$ and where the regularization parameter $\lambda_k$ has (i) absorbed possible multiplicative constants and (ii) is allowed to change at each iteration. It can be shown that, \revmsl{for fixed $\lambda_k=\lambda$}, if one fully solves each \revision{subproblem in the sequence}, the computed solution converges to the solution of the original problem, for a proof, see e.g., \cite{daubechies2010}. 

\subsection{Krylov methods for $\ell_p$ regularization}\label{Sec:Krylov_weight}

Traditionally, each subproblem \eqref{recursion} in the recursion is solved with high accuracy: for example, using an iterative method, where a new Krylov subspace is constructed at each outer iteration. These methods can produce very accurate solutions, however, the inner-outer loops of iterations can be very computationally expensive. \revision{More recently, new methods have been developed to alleviate the computational burden of IRW schemes wherein 
the subproblems are only \revmsl{approximately} solved, in the sense that each problem is projected in a subspace of increasing dimension, and therefore only one Krylov subspace is built in the algorithm.} \st{More recently, and to alleviate the computation burden of IRW schemes, new methods have appeared and where the subproblems are only partially solved, in the sense that each problem is projected in a subspace of increasing dimension, and therefore only one Krylov subspace is built in the algorithm.}
These methods differ in the way they build \revision{their subspaces}, either using generalized or flexible subspaces, as cited in Section \ref{sec:Intro}. More on flexible methods will be presented in Section~\ref{Sec:Flexible}.

\subsection{Flexible Krylov methods for $\ell_p$ regularization}\label{Sec:Flexible}

Recently, methods to solve \eqref{recursion} have been developed which avoid the inner-outer scheme of iterations by constructing a single (flexible) Krylov subspace. First, consider the regularization matrix $\mathbf{L}$ in \eqref{recursion} to be invertible. Then, \eqref{recursion} can be re-written using a standard change of variables as 
\begin{equation}\label{recursion_22}
\mathbf{z}_{k} = \argmin_\mathbf{z} \left\|\mathbf{A} (\mathbf{W}_k \mathbf{L})^{-1}\mathbf{z} - \mathbf{b}\right\|_2^2+ \lambda_k \lVert\mathbf{z}\rVert_2^2, \quad \text{where } \mathbf{x}_k = (\mathbf{W}_k \mathbf{L})^{-1}\mathbf{z}_k.
\end{equation}

This equivalent problem is said to be in standard form since the regularization matrix in \eqref{recursion_22} is the identity matrix. If $\mathbf{L}$ is not invertible, for example, in the case of discretized TV with general boundary conditions, one can still write an equivalent problem in standard form involving the $A$-weighted pseudo-inverse \cite{Elden1982pseudoinverse}, but these methods are usually computationally demanding and not in the scope of this paper, see for example \cite{Hansen2007smoothingnorm, Hansen2013Aweighted, Gazzola2021Edge}. Note that the matrix $(\mathbf{W}_k \mathbf{L})^{-1}$ can be interpreted as an iteration-dependent preconditioner. Since the regularization matrix in a variational problem can be associated to prior information on the solution, $(\mathbf{W}_k \mathbf{L})^{-1}$ \revision{may also be called a} prior-conditioner \cite{Calvetti2007priorconditioning}. This can be formalized by using a Bayesian interpretation of the inverse problem \eqref{eq:original}, see, e.g. \cite{Calvetti2018Bayesian}.

A very powerful framework to deal with sequences of problems of the form \eqref{recursion_22} is to update the weights every time a new approximation of the solution becomes available, i.e., at each iteration, and to consider flexible Krylov subspace methods which allow for iteration-dependent preconditioning. In this paper, we focus on FGMRES and FLSQR, which are based on flexible extensions of the Arnoldi and Golub-Kahan bidiagonalization procedures to build basis vectors \revision{that include} flexible preconditioning.

Assume we have iteration-dependent right preconditioning matrices $(\mathbf{W}_k \bfL)^{-1}$. \revision{If break-down of the flexible Arnoldi (FA) algorithm has not occurred, then the relation at each iterate $k$ for $k<n$ for a given $\mathbf{A} \in \mathbb{R}^{n \times n}$ and $\mathbf{b}\in \mathbb{R}^{n}$ is given by} \st{Then, if no break-down of the algorithm has happened, the flexible Arnoldi (FA) method for a given $\mathbf{A} \in \mathbb{R}^{n \times n}$ and $\mathbf{b}\in \mathbb{R}^{n}$, updates the following relation at each iteration $k < n$:}
\begin{equation} \label{eq:flexArnoldi}
    \mathbf{AZ}_k = \mathbf{V}_{k+1}\mathbf{H}_{k+1},
\end{equation}
where the columns of $\mathbf{Z}_k=[{(\mathbf{W}_1 \mathbf{L})^{-1} \mathbf{v}_1,..., (\mathbf{W}_k \mathbf{L}) ^{-1} \mathbf{v}_k}] \in \mathbb{R}^{n \times k}$ span the solution search space\st{for the solution}, $\mathbf{H}_{k+1} \in \mathbb{R}^{ (k+1) \times k}$ is upper Hessenberg, and $\mathbf{V}_{k+1} \in \mathbb{R}^{n \times (k+1)}$ has orthonormal columns \revision{whose first column} corresponds to the normalized right-hand side $\mathbf{b}/\|\mathbf{b}\|_2$. A particular implementation of one step of this process can be observed in Algorithm \ref{flexibleArnoldi_step}. This method was first introduced in \cite{Saad1993FGMRES}, and has been widely used in the context of reweighted schemes \cite{Gazzola2014GAT,JulianneSilvia,Gazzola2019TV,Gazzola2021IRW,Chung_2024}.  

\revision{When $\mathbf{A}$ is non-square and given the same preconditioning matrices,} \st{On the other hand, given the same preconditioning matrices, }the flexible Golub-Kahan (FGK) decomposition with $\mathbf{A} \in \mathbb{R}^{m \times n}$ and $\mathbf{b}\in \mathbb{R}^{m}$ updates the following relations at each iteration  $k < \min \{ m,n\}$:
\begin{equation}\label{eq:partialFGKB}
 \mathbf{AZ}_k = \mathbf{U}_{k+1} \mathbf{M}_{k+1} \quad \text{and} \quad \mathbf{A}^T \mathbf{U}_{k+1}=\mathbf{V}_{k+1} \mathbf{S}_{k+1}.
\end{equation}
Here, the columns of $\mathbf{Z}_k=[{(\mathbf{W}_1 \mathbf{L})^{-1} \mathbf{v}_1,..., (\mathbf{W}_k \mathbf{L})^{-1} \mathbf{v}_k}] \in \mathbb{R}^{n \times k}$ span the solution search space. Moreover, $\mathbf{M}_{k+1} \in \mathbb{R}^{ (k+1) \times k}$ is upper Hessenberg, $\mathbf{U}_{k+1} \in \mathbb{R}^{m \times (k+1)}$ has orthonormal \revision{columns whose first column corresponds to the normalized right-hand side $\mathbf{b}/\|\mathbf{b}\|_2$}, $\mathbf{S}_{k+1} \in \mathbb{R}^{(k+1) \times (k+1)}$ is upper triangular, and $\mathbf{V}_{k+1} \in \mathbb{R}^{n \times (k+1)}$ has orthonormal columns. A particular implementation of one step of FGK can be observed in Algorithm \ref{flexibleGKB_step}. 

\begin{algorithm} 
\caption{[$\mathbf{z}_k$, $\mathbf{h}_{\revmsl{k}}$, $\mathbf{v}_{k+1}$] = step\_FA($\mathbf{A}, \mathbf{V}_{k}, \mathbf{W}_{k} \mathbf{L}$)} \label{flexibleArnoldi_step}
$\%$ step of the Flexible Arnoldi (FA) algorithm  \\
\text{\bf{Input:}} $\mathbf{A} \in \mathbb{R}^{n \times n},\, \mathbf{V}_{k} \in \mathbb{R}^{n \times k}, \mathbf{W}_{k} \mathbf{L}$\\
\text{\bf{Output:}} $\mathbf{z}_{k} \in \mathbb{R}^{n},\, \mathbf{h}_{\revmsl{k}} \in \mathbb{R}^{k+1}, \;\mbox{and}\; \mathbf{v}_{k+1} \in \mathbb{R}^{n}$\\
    $\mathbf{z}_k = (\mathbf{W}_{k} \mathbf{L})^{-1}  \mathbf{v}_{k}$ \\
    $\mathbf{q} = \mathbf{A}\mathbf{z}_k$ \\
    \For{$i = 1,\dots,k$}{
        $[\mathbf{h}_{\revmsl{k}}]_{i} = \mathbf{q}^T\mathbf{v}_i$ and $\mathbf{q} = \mathbf{q} - [\mathbf{h}_{\revmsl{k}}]_{i}\mathbf{v}_i$ \\
    }
    $[\mathbf{h}_{\revmsl{k}}]_{k+1} = \|\mathbf{q}\|_2$ \label{hii}\\
    $\mathbf{v}_{k+1} = \mathbf{q}/[\mathbf{h}_{\revmsl{k}}]_{k+1}$ 
\end{algorithm}

\begin{algorithm} 
\caption{[$\mathbf{z}_k, \mathbf{s}_{k}, \mathbf{v}_{k+1}, \mathbf{m}_{\revmsl{k}}, \mathbf{u}_{k+1}$] = step\_FGK($\mathbf{A}, \mathbf{U}_{k}, \mathbf{V}_{k-1}, \mathbf{W}_{k} \mathbf{L}$)}\label{flexibleGKB_step}
$\%$ step of the Flexible Golub-Kahan (FGK) algorithm  \\
\text{\bf{Input:}} $\mathbf{A} \in \mathbb{R}^{m \times n},\,\mathbf{U}_{k} \in \mathbb{R}^{m \times k}, \mathbf{V}_{k-1} \in \mathbb{R}^{n \times (k-1)}\mathbf{W}_{k} \mathbf{L}$ 
\\
\text{\bf{Output:}} $\mathbf{z}_{k} \in \mathbb{R}^{n},\, \mathbf{s}_{k} \in \mathbb{R}^{k},\, \mathbf{v}_{k} \in \mathbb{R}^{n},\, \mathbf{m}_{\revmsl{k}} \in \mathbb{R}^{k+1}, \;\mbox{and}\; \mathbf{u}_{k+1} \in \mathbb{R}^{m}$\\
$\mathbf{w}= \mathbf{A}^T \mathbf{u}_k$ \\
    \For{$i = 1,\dots,k-1$}{
        $[\mathbf{s}_k]_{i} = \mathbf{w}^T \mathbf{v}_i$ and $\mathbf{w} = \mathbf{w} - [\mathbf{s}_k]_{i} \mathbf{v}_i$ \\
    }
$[\mathbf{s}_k]_{k}=\|\mathbf{w}\|_2$ \label{sii}\\
$\mathbf{v}_{k}=\mathbf{w}/[\mathbf{s}_k]_{k}$\\
$\mathbf{z}_k={(\mathbf{W}_k \mathbf{L})^{-1}}\,\mathbf{v}_k$\\
$\mathbf{q} = \mathbf{Az}_k$ \\
    \For{$i = 1,\dots,k$}{
        $[\mathbf{m}_{\revmsl{k}}]_{i} = \mathbf{q}^T \mathbf{u}_i$ and $\mathbf{q} = \mathbf{q} - [\mathbf{m}_{\revmsl{k}}]_{i} \mathbf{u}_i$ \\
    }
$[\mathbf{m}_{\revmsl{k}}]_{k+1}=\|\mathbf{q}\|_2$ \label{mii} \\ 
$\mathbf{u}_{k+1}=\mathbf{q}/[\mathbf{m}_{k+1}]_{k+1}$ 
\end{algorithm}

Note that both algorithms might break down  at iteration $k$ if, for FA, $[\mathbf{H}]_{k+1,k} = 0$ (line \ref{hii} of Algorithm \ref{flexibleArnoldi_step}), or in FGK,  $[\mathbf{M}]_{k+1,k} = 0$ or $[\mathbf{S}]_{k,k} = 0$ (lines \ref{mii} and \ref{sii} of Algorithm \ref{flexibleGKB_step}). However, there is empirical evidence that this is not a likely scenario, see, e.g. \cite{JulianneSilvia}. 

One can now treat the recursion \label{recursion_2} by projecting each problem into the corresponding (nested) space of increasing dimensions spanned by the columns of $\mathbf{Z}_k$ produced by the FA method in \eqref{eq:flexArnoldi} or the FGK in \eqref{eq:partialFGKB}. In the literature, \revision{there are various ways to treat the regularization term. However, when considering the projected problem, all methods can be summarized by the following general form:}
\begin{equation}\label{projected_general_FKS}
    \bfy_k = \argmin_{\bfy \in \mathbb{R}^{k}} \left\{
    \left\|\bfT_k \bfy - \|\bfb\|_2 \bfe_1\right\|_2^2 +  \lambda_k \|\bfP \bfy\|^2_2
    \right\}; 
    \quad \quad \bfx_k = \bfZ_k \bfy_k.
\end{equation}
The methods can be \revision{sub-classified as follows}: 
\begin{itemize}
    \item \underline{Flexible Krylov methods}: these rely on the regularization induced by the structure of the search spaces, without explicit regularization. At each iteration, they solve the following projected problem:
    \begin{equation}\label{projected_FKS}
    \bfy_k = \argmin_{\bfy \in \mathbb{R}^{k}} \left\{
    \left\|\bfT_k \bfy - \|\bfb\|_2 \bfe_1 \right\|_2^2 \right\}; 
    \quad \quad \bfx_k = \bfZ_k \bfy_k,
\end{equation}
    where $\bfT_k$ depends on the choice of the flexible Krylov basis:  flexible  GMRES (FGMRES) considers $\bfT_k = \bfH_{\revmsl{k+1}}$ as defined in \eqref{eq:flexArnoldi} and flexible LSQR (FLSQR) considers $\bfT_k = \bfM_{\revmsl{k+1}}$ as defined in \eqref{eq:partialFGKB}, in this context, see \cite{JulianneSilvia}. Since these methods do not consider any explicit regularization, they converge to the solution of the least-squares problems associated with the noisy right-hand side\revision{, which, does not typically provide a meaningful solution.} \revision{As such, they need to be terminated at an early iteration.} For fixed preconditioning, i.e., $\mathbf{W}_k=\mathbf{W}$, FGMRES and FLSQR reduce to standard preconditioned GMRES and LSQR, the latter also being equivalent to CG applied to the normal equation with split preconditioning $(\mathbf{W} \mathbf{L})^{-1/2}$. 

    \item \underline{Hybrid flexible Krylov methods}: these follow a project-then-regularize scheme, where Tikhonov regularization is added to the projected problem 
    \cite{JulianneSilvia}. In this case, the following minimization problem is solved at each iteration:
    \begin{equation}\label{projected_HFK}
    \bfy_k = \argmin_{\bfy \in \mathbb{R}^{k}} \left\{
    \left\|\bfT_k \bfy - \|\bfb\|_2 \bfe_1 \right\|_2^2 +  \lambda_k \|\bfy\|^2_2
    \right\}; 
    \quad \quad \bfx_k = \bfZ_k \bfy_k,
\end{equation}
    where $\lambda_k$ can be chosen efficiently at each iteration by using regularization parameter choice criteria on the projected problem. As before, hybrid flexible GMRES (H-FGMRES) considers $\bfT_k = \bfH_{\revmsl{k+1}}$, and hybrid flexible LSQR (H-FLSQR) considers $\bfT_k = \bfM_{\revmsl{k+1}}$. The computational cost of these methods is comparable to that of the standard flexible Krylov methods, while they avoid semi-convergence by incorporating explicit regularization. However, these methods converge to the solution of the original least squares problems with added (standard) Tikhonov regularization instead of the $\ell_p$  regularization term. 
    
    \item \underline{Iteratively Reweighted Flexible methods}: these  consider a regularize-then-project approach, in the sense that both the least squares associated to the fit-to-data term and the regularization term are projected onto the solution space. Therefore, at each iteration, they solve:
    \begin{equation}\label{projected_IRW_FKS}
    \bfy_k = \argmin_{\bfy \in \mathbb{R}^{k}} \left\{
    \left\|\bfT_k \bfy - \|\bfb\|_2 \bfe_1 \right\|_2^2 +  \lambda_k \|\bfW_k \bfL \revmsl{\bfZ_k} \bfy\|^2_2
    \right\}; 
    \quad \quad \bfx_k = \bfZ_k \bfy_k.
\end{equation}
    
    \revision{Here, the regularization matrix arising from the majorization of the $\ell_p$ norm, $\bfW_k \bfL $, is included.}\st{This means, including the original regularization matrix arising from the majorization of the $\ell_p$ norm: $\bfW_k \bfL$.}  In practice, this is done by considering $\mathbf{P}= \bfR_{\revmsl{k}}$ in \eqref{projected_general_FKS}, where $\bfR_{\revmsl{k}}$ is obtained doing a tall and skinny economic QR factorization of the matrix $\bfW_k \bfL \revmsl{\bfZ_k}$. 
    In the same fashion to hybrid methods, the regularization parameter $\lambda_k$ can be chosen efficiently at each iteration using only information from the projected problem. As before, IRW-FGMRES considers $\bfT_k = \bfH_{\revmsl{k+1}}$, and IRW-FLSQR considers $\bfT_k = \bfM_{\revmsl{k+1}}$,  \cite{Gazzola2021IRW}.
    This class of methods converge to the solution of (the smoothed version of) problem \eqref{eq:l1}, at the cost of an extra (efficient) economic QR factorization of a tall and skinny matrix per iteration. 
\end{itemize}

\revision{For completeness, we provide the flexible Arnoldi framework in Algorithm \ref{alg:FA-methods} and the flexible Golub-Kahan in Algorithm \ref{alg:FGK-methods}. We note that in all scenarios that we are building an appropriate space for $\mathbf{x}$ and not for $\mathbf{z}$, since the preconditioning $(\bfW_k \bfL)^{-1}$ is already embedded in the subspace.}
\st{Refer to Algorithm 3, for flexible Arnoldi methods.
Note that in all scenarios, we are building an appropriate space for $\mathbf{x}$ and not for $\mathbf{z}$, since the preconditioning $(\bfW_k \bfL)^{-1}$ is already embedded in the subspace.}

\begin{algorithm}[!ht]
\caption{Flexible Arnoldi based Krylov subspace methods} \label{alg:FA-methods}
\text{\bf{Input:}} $\mathbf{A}$, $\mathbf{b}$, $\mathbf{L}$, $\mathbf{x}_0$, maximum number of iterations $\mathrm{kmax}$ \\
\text{\bf{Output:}} Approximate solution $\mathbf{x}_{\mathrm{kmax}}$ \\
Initialize weights $\mathbf{W}_1= \mathbf{I}$\\
$\mathbf{r}_0 = \mathbf{b} - \mathbf{A} \mathbf{x}_{0}$ \\
Initialize Flexible Arnoldi (FA) $\mathbf{H}_0=[], \, \mathbf{Z}_0=[], \mathbf{V}_1=[\mathbf{r}_0 /\|\mathbf{r}_0\|_2] $\\
\For{$k=1,\dots, \mathrm{kmax}$}{
    Compute [$\mathbf{z}_k$, $\mathbf{h}_{\revmsl{k}}$, $\mathbf{v}_{k+1}$] = step\_FA($\mathbf{A}, \mathbf{V}_{k}, \mathbf{W}_{k} \mathbf{L}$) \Comment{Algorithm \ref{flexibleArnoldi_step}} \label{alg_step_fa1} \\
    Update $\mathbf{Z}_k=[\mathbf{Z}_{k-1}, \mathbf{z}_{k}], \,\mathbf{H}_k=\begin{bmatrix} \mathbf{H}_{k-1} & \mathbf{h}_k\\ \mathbf{0} & \phantom{\,} \end{bmatrix},  \, \mathbf{V}_{k+1}=[\mathbf{V}_{k}, \mathbf{v}_{k+1}] $ \label{alg_step_fa2}\\
    Compute coefficients $\displaystyle \mathbf{y}_k$ according to \eqref{projected_general_FKS}  \Comment{compute update} \label{alg_step_fa}
    \\
    $\mathbf{x}_k = \mathbf{Z}_k \mathbf{y}_k$ 
    \\
    $\mathbf{W}_{k+1}=\widetilde{W}(\mathbf{L} \mathbf{x}_k)$ according to \eqref{eq:l1weightssmooth} \Comment{update weights} \\
    \If{stopping criterion}{$\mathrm{kmax} = k$\\ $\mathrm{break}$}}
\end{algorithm}

\begin{algorithm}[!ht]
\caption{Flexible Golub-Kahan based Krylov subspace methods} \label{alg:FGK-methods}
\text{\bf{Input:}} $\mathbf{A}$, $\mathbf{b}$, $\mathbf{L}$, $\mathbf{x}_0$, maximum number of iterations $\mathrm{kmax}$ \\
\text{\bf{Output:}} Approximate solution $\mathbf{x}_{\mathrm{kmax}}$ \\
Initialize weights $\mathbf{W}_1= \mathbf{I}$\\
$\mathbf{r}_0 = \mathbf{b} - \mathbf{A} \mathbf{x}_{0}$ \\
Initialize Flexible Golub-Kahan (FGK) $\mathbf{Z}_0=[], \,\mathbf{S}_0=[], \,  \mathbf{V}_0=[], \, \mathbf{M}_0=[], \, \mathbf{U}_1=[\mathbf{r}_0 /\|\mathbf{r}_0\|_2] $\\
\For{$k=1,\dots, \mathrm{kmax}$}{
    Compute [$\mathbf{z}_k, \mathbf{s}_{k}, \mathbf{v}_{\revmsl{k}}, \mathbf{m}_{\revmsl{k}}, \mathbf{u}_{k+1}$] = step$\_$FGK($\mathbf{A}, \mathbf{U}_{k}, \mathbf{V}_{k-1}, \mathbf{W}_{k} \mathbf{L}$) \Comment{Algorithm \ref{flexibleGKB_step}} \label{alg_step_fgkb1} \\
    Update $\mathbf{Z}_k=[\mathbf{Z}_{k-1}, \mathbf{z}_{k}],$ \,$\mathbf{S}_k=\begin{bmatrix} \mathbf{S}_{k-1} & \mathbf{s}_k\\ \mathbf{0} & \phantom{\,} \end{bmatrix}$,  \, $\mathbf{V}_{\revmsl{k}}=[\mathbf{V}_{\revmsl{k-1}}, \mathbf{v}_{\revmsl{k}}]$  \\
     Update $\mathbf{U}_{k+1}=[\mathbf{U}_{k}, \mathbf{u}_{k+1}], \,\mathbf{M}_{\revmsl{k}}=\begin{bmatrix} \mathbf{M}_{\revmsl{k-1}} & \mathbf{m}_{\revmsl{k}}\\ \mathbf{0} & \phantom{\,} \end{bmatrix}$ 
     \label{alg_step_fgkb2}\\
    Compute coefficients $\displaystyle \mathbf{y}_k$ according to \eqref{projected_general_FKS}  \Comment{compute update} \label{alg_step_gkb}
    \\
    $\mathbf{x}_k = \mathbf{Z}_k \mathbf{y}_k$ 
    \\
    $\mathbf{W}_{k+1}=\widetilde{W}(\mathbf{L} \mathbf{x}_k)$ according to \eqref{eq:l1weightssmooth} \Comment{update weights} \\
    $\mathbf{r}_k = \mathbf{b} - \mathbf{A} \mathbf{x}_{k-1}$ \\
    \If{stopping criterion}{$\mathrm{kmax} = k$\\ $\mathrm{break}$}
    }
\end{algorithm}

\section{IRN as an iterative refinement scheme}\label{Sec-IRNasIR}

In this section, we recast the iteratively reweighted norm (IRN) framework, detailed in Subsection \ref{Sec-IRN}, as an iterative refinement \revision{scheme}\st{method}, which \revision{we outline in Section \ref{Sec-IR}}.\st{ is further described in Subsection 3.1.} \revision{We clarify that this section does not describe a new algorithm, nor promote any particular implementation, rather, our focus here is to provide a new perspective on IRN schemes which serves as the baseline for the new methods we introduce in the following sections.} \st{Note that this section is not meant to describe a new algorithm, or promote any particular implementation, but rather to give a new perspective on IRN schemes which served as the starting point in the development of the new methods that are detailed in the following sections.}

\subsection{Background on iterative refinement}\label{Sec-IR}
Classically, iterative refinement is an algorithmic scheme designed to improve the accuracy of the solution to linear systems by computing approximate updates \revision{to the solution at each iteration \cite{Wilk48}.} To compute the $k^{th}$ approximate solution, $\mathbf{x}_{k}$, to the available linear system \eqref{eq:original}, the $(k-1)^{th}$ residual $\mathbf{r}_{k-1}=\mathbf{b}-\mathbf{Ax}_{k-1}$ is first computed. Then, the correction equation $\mathbf{Ah}=\mathbf{r}_{k-1}$ is solved, and the current solution is updated by $\mathbf{x}_{k}=\mathbf{x}_{k-1}+\mathbf{h}_{k-1}$. The process may be repeated as necessary until a termination criterion is met.

\subsection{IRN schemes as iterative refinement}
Recall that IRW schemes have an inherent inner-outer \revision{algorithmic structure}, since they require updating the regularization matrix as soon as a new approximation of the solution is available. \revision{As such, one can recast them through the lens of iterative refinement since at each outer iteration the IRN scheme computes the solution of the appropriate recursion equation \eqref{majorant}.} \st{This needs to be done carefully to assure the updated systems are equivalent, in the sense that, at each outer iteration, the IRN scheme is computing the solution of the appropriate recursion equation (majorant).} In the most basic form, this corresponds to Algorithm \ref{alg:IRIRW}. 
In this case, we say that we are doing `warm restarts' in the sense that we compute an update of the solution at each possible (outer) iteration, where the initial guess is $\bfx_{k-1}$ so that $\bfx_{k} = \bfx_{k-1} +\bfh_{k}$ and we optimize to obtain $\bfh_{k}$. This is in contrast to `cold restarts', where a new solution is computed from scratch at each possible (outer) iteration, i.e., by considering the initial guess for the solution to be $\bf0$. 
We observe that, from a practical point of view, solving for the update $\mathbf{h}_k$ (and not the solution $\mathbf{x}_k$) leads to convergence with a smaller amount of inner iterations, so this is very convenient memory-wise.

\begin{algorithm}[!ht]
\caption{Iterative refinement for IRN schemes} \label{alg:IRIRW}
\text{\bf{Input:}} $\mathbf{A}$, $\mathbf{b}$, $\mathbf{x}_0$, maximum number of iterations $\mathrm{kmax}$ 
\\
\text{\bf{Output:}} Approximate solution $\mathbf{x}_{k}$  \\
Initialize weights $\mathbf{W}_1= \mathbf{I}$\\
\For{$k=1,\dots,\mathrm{kmax}$}{
    $\mathbf{r}_k = \mathbf{b} - \mathbf{A} \mathbf{x}_{k-1}$ \\
    $\displaystyle  \mathbf{h}_k = \argmin_{\mathbf{h}} \left\{ \left\lVert\begin{bmatrix} \mathbf{A} \\ \sqrt{\lambda_k} \mathbf{W}_k \mathbf{L} \end{bmatrix} \mathbf{h} - \begin{bmatrix} \mathbf{r}_k \\ - \sqrt{\lambda_k} \mathbf{W}_k \mathbf{L} \, \mathbf{x}_{k-1} \end{bmatrix}\right\rVert_2 \right\}$ \Comment{compute update} \label{alg_step}
    \\
    $\mathbf{x}_{k}=\mathbf{x}_{k-1} + \mathbf{h}_{k}$ \\
    $\mathbf{W}_k=\widetilde{W}( \mathbf{L} \mathbf{x}_k)$ according to \eqref{eq:l1weightssmooth} \Comment{update weights}\\
    \If{stopping criterion}{$\mathrm{kmax} = k$\\ $\mathrm{break}$}
}
\end{algorithm}

In practice, step \ref{alg_step} of Algorithm \ref{alg:IRIRW} can be implemented using a Krylov method, so that, at the \revmsl{$i^{th}$ }iteration \st{of the Krylov method} for the \revmsl{$k^{th}$} outer iteration, we solve

\begin{eqnarray} 
\mathbf{h}^{(i)}_k &=& \arg \min_{\mathbf{h} \in \mathcal{K}_i=\mathcal{R}(\mathbf{V}_i)}  \Bigg\{ \Bigg\lVert
\underbrace{\begin{bmatrix} \mathbf{A} \\ \sqrt{\lambda_{k}} \mathbf{W}_k \revmsl{\bfL}\end{bmatrix}}_{\widehat{\bfA}_k}
\mathbf{h} - \underbrace{\begin{bmatrix} \mathbf{r}_k \\ - \textcolor{black}{\sqrt{\lambda_{k}}}\mathbf{W}_k \revmsl{\bfL}\, \mathbf{x}_{k-1} \end{bmatrix}}_{\widehat{\bfr}_k}\Bigg\rVert_2 \Bigg\} \nonumber\\ 
&=& \bfV_i \cdot \arg \min_{\mathbf{y} \in \mathbb{R}^i} \left\{ \left\lVert\begin{bmatrix} \mathbf{A} \bfV_i \\ \sqrt{\lambda_{k}} \mathbf{W}_k \revmsl{\bfL} \bfV_i\end{bmatrix} \mathbf{y} - \begin{bmatrix} \mathbf{r}_k \\ - \textcolor{black}{\sqrt{\lambda_{k}}}\mathbf{W}_k \revmsl{\bfL}\, \mathbf{x}_{k-1} \end{bmatrix}\right\rVert_2 \right\},
\end{eqnarray} 
where the columns of $\bfV_i$ correspond to the basis vectors generated by the Golub-Kahan bidiagonalization (GKB) method applied to $\widehat{\bfA}_k$ (note that this is not a square matrix even for square ${\bfA}$) and $\widehat{\bfr}_k$, after $i$ iterations \cite{GolubVanLoan}. \revmsl{To finish, $\bfh_k$ in step \ref{alg_step} of Algorithm \ref{alg:IRIRW} is taken to be $\bfh^{(i)}_k$ for $i$ corresponding to the stopping iteration}. 
Note that in this case, the regularization parameter is defined as a part of the system matrix and right-hand side, so it needs to be set in advance of the iterations. Moreover, another drawback of this method is that computing the update \revmsl{still} corresponds to the inner iterations of an inner-outer scheme, so it can be a computationally expensive task. 


\section{Combining IRW flexible Krylov methods and iterative refinement}\label{Sec-new-method}
In this section, we propose two new solvers: IR-FGMRES and IR-FLSQR, and two corresponding modified versions \revision{that include a subspace correction:} \st{including a subspace correction:} CIR-FGMRES and CIR-FLSQR. All four methods are based on flexible Krylov subspaces where the iteration-dependent preconditioning is related to the IRN weights that exploit the structure of iterative refinement schemes. In particular, these methods are very efficient since they avoid nested loops of iterations and allow for the natural implementation of (i) (warm) restarting - which can overcome stagnation and also reduce memory requirements - and (ii) automatic parameter selection.

\subsection{Iteratively Refined and (restarted) flexible Krylov methods}
Recent work on linear system solvers using Krylov projection methods within an iterative refinement scheme considered the construction of a (single) subspace which \revmsl{is} augmented for every correction step \cite{Onisk2022Arnoldi}. \revision{This approach is advantageous since it can avoid having a nested loop of iterations.} \st{This is very efficient compared to the previous methods, since it avoids having a nested loop of iterations.} In this work, we adopt this idea in the context of IRN. However, differently from \cite{Onisk2022Arnoldi}, we consider a flexible Krylov subspace that can handle the changes in the regularization matrix. Moreover, and differently from other flexible Krylov methods for IRN such as \cite{Gazzola2021IRW, JulianneSilvia}, our new methods naturally allow for (warm) restarts. This is crucial to overcome two main problems associated to the latter methods: we can reduce the memory requirements and overcome stagnation of the solution; while maintaining their strengths: efficiency and automatic choice of the regularization parameters throughout the iterations. 

A general description for both the new Iterative Refinement Flexible GMRES (IR-FGMRES) and LSQR (IR-FLSQR), can be found in Algorithm \ref{alg:FIRKS}, where we use the flexible Arnoldi and flexible Golub-Kahan decompositions described in Section \ref{Sec:Flexible}.
\begin{algorithm}[!ht]
\caption{Iterative Refinement (restarted) Flexible Krylov subspace methods} \label{alg:FIRKS}
\text{\bf{Input:}} $\mathbf{A}$, $\mathbf{b}$, $\mathbf{x}_0$, maximum number of iterations $\mathrm{kmax}$ \\
\text{\bf{Output:}} Approximate solution $\mathbf{x}_{\mathrm{kmax}}$ \\
Initialize weights $\mathbf{W}_1= \mathbf{I}$\\
$\mathbf{r}_0 = \mathbf{b} - \mathbf{A} \mathbf{x}_{0}$ \\
Set $\mathbf{V}_1 = \mathbf{r}_0 / \|\mathbf{r}_0 \|_2$ (Arnoldi) or $\mathbf{U}_1 = \mathbf{r}_0 / \|\mathbf{r}_0 \|_2$ (Golub-Kahan) \\ 
$i = 1$\\
\For{$k=1,\dots, \mathrm{kmax}$}{
    \If{inner (restart) stopping criterion}{
    Set $\mathbf{V}_1 = \mathbf{r}_k / \|\mathbf{r}_k \|_2$ (Arnoldi) or $\mathbf{U}_1 =  \mathbf{r}_k/ \|\mathbf{r}_k \|_2$ (Golub-Kahan) \label{addZ_FIRKS}\\ $i = 1$\\
    }
    Update $\mathbf{Z}_i$ (and other relevant matrices) following lines \ref{alg_step_fa1}-\ref{alg_step_fa2} of Algorithm \ref{alg:FA-methods} (Arnoldi) or lines \ref{alg_step_fgkb1}-\ref{alg_step_fgkb2} of Algorithm  \ref{alg:FGK-methods} (Golub-Kahan)
     \\
    $\displaystyle \mathbf{y}_i = \argmin_{\mathbf{y \in \mathbb{R}^{i}}} \left\{ \left\lVert\begin{bmatrix} \mathbf{A} \mathbf{Z}_i \\ \sqrt{\lambda_i} \mathbf{W}_k \mathbf{L} \mathbf{Z}_i \end{bmatrix} \mathbf{y} - \begin{bmatrix} \mathbf{r}_k \\ - \sqrt{\lambda_i} \mathbf{W}_k \mathbf{L} \, \mathbf{x}_{k-1} \end{bmatrix}\right\rVert_2 \right\}$ \Comment{compute update} \label{alg_step_FIRKS}
    \\
    $\mathbf{x}_k = \mathbf{x}_{k-1} +\mathbf{Z}_i \mathbf{y}_i$ 
    \\
    $\mathbf{W}_k=\widetilde{W}(\bfL \mathbf{x}_k)$ according to \eqref{eq:l1weightssmooth} \Comment{update weights} \\
    $\mathbf{r}_{k+1} = \mathbf{b} - \mathbf{A} \mathbf{x}_{k}$ \\
    \If{outer stopping criterion}{$\mathrm{kmax} = k$\\ $\mathrm{break}$}
    $i = i+1$
}
\end{algorithm}
Note that line \ref{alg_step_FIRKS} of Algorithm \ref{alg:FIRKS} can be computed efficiently by using the following projection:
\begin{eqnarray} 
\displaystyle \revmsl{\bfy_i} &=& \argmin_{\mathbf{y} \in \mathbb{R}^{\revmsl{i}}} \left\{ \left\lVert \mathbf{A} \bfZ_{\revmsl{i}}  \mathbf{y} - \mathbf{r}_k \right\rVert_2^2  + \lambda_{\revmsl{i}} \|  \mathbf{W}_k \mathbf{L} \revmsl{\bfZ_i} \mathbf{y} - \mathbf{W}_k \mathbf{L} \, \mathbf{x}_{k-1} \|_2^2\right\} \\
&=&  \argmin_{\mathbf{y} \in \mathbb{R}^{\revmsl{i}}} \left\{ \left\lVert \bfT_k  \mathbf{y} - \bfX_k^T\bfr_k \right\rVert_2^2  + \lambda_{\revmsl{i}} \|  \bfR_k \mathbf{y} - \bfQ_k^T \revmsl{\hat{\bfx}}_{k-1} \|_2^2\right\},
\label{eq:update}
\end{eqnarray} 
\revmsl{where $\hat{\bfx}_{k-1} = \bfW_k \bfL \bfx_{k-1}$.} Here, we have projected the first term using either the Arnoldi relation \eqref{eq:flexArnoldi}, so that  $\bfT_k = \bfH_{\revmsl{k+1}}$ and $\bfX_k=\bfV_k$, or the Golub-Kahan relation \eqref{eq:partialFGKB}, so that $\bfT_k = \bfM_{\revmsl{k+1}}$ and $\bfX_k=\bfU_k$; and an economic QR factorization of the tall and skinny matrix $\bfQ_k \bfR_k = \mathbf{W}_k \mathbf{L} {\revmsl{\bfZ_i}} $, which can be done efficiently since the number of iterations $i$ before a restart is typically small. Moreover, since \revmsl{the} the problem \eqref{eq:update} is \revmsl{also} small, and the projection does not depend on the choice of the regularization parameter, one can seamlessly use \revision{standard regularization parameter selection methods} \st{standard regularization parameter choice} to compute a suitable $\lambda_{k}$, and this parameter can change at each iteration. We provide a short discussion on regularization parameter \revision{schemes}\st{choices}, including criteria to restart the Krylov subspaces, in Section \ref{sec:reg_parag}.


\subsection{Corrected iteratively refined flexible IRN}

One of the key properties of our \revision{proposed methods is that the solution subspaces that are built are a good fit for the measurements. Doing so leverages the properties of Krylov methods and requires the incorporation of prior information on the solution through the priorconditioning of the basis vectors. If, however, a solution at iteration $k$ when a restart is necessary is not yet a good approximation, a method might stagnate.} A simple way of avoiding this behavior is to augment the corresponding flexible Krylov subspace with the last approximate solution computed before the restart. 

\st{There are different ways in which one can incorporate a given direction in the solution space; sometimes with some overlap in the naming conventions.} 
Subspace augmentation (or enrichment) considers \revision{the explicit addition of a set of vectors to a Krylov subspace to `augment' the solution subspace, see e.g.,} \st{adding explicitly a given set of external vectors to the Krylov subspace to `augment' the solution subspace, see e.g.} \cite{Chapman1997Augmented,Baglama2007AugmentedGMRES,Hansen2019AugmentedLSQR, BDOR25}. In general, these vectors can come from prior knowledge of the solution, solutions of closely related systems, or previous iterations before a restart; and they can either enforce the Krylov subspace to be orthogonal to the augmented space or not. Moreover, in subspace recycling, the Krylov subspace iterations keep orthogonality to the image of the recycled space, so that the Krylov space complements the recycled space. This allows for the choice of optimal recycled spaces, but it requires additional computational cost, see the review \cite{Soodhalter2020recycling}. 

In this paper, we follow an augmentation-type approach where the Krylov subspace is kept orthogonal to the added vector; but it is built from the system matrix and the residual without keeping orthogonality to the image of the added vector (and therefore it is not considered a recycling scheme). In particular, after each restart, and given the previous solution $\mathbf{x}_{k-1}$, we consider $\bar{\mathbf{z}}_1 = \mathbf{x}_{k-1}/\|\mathbf{x}_{k-1}\|$, and we use the (normalized version of the) vector $\left(\mathbf{I}-\bar{\mathbf{z}}_1\bar{\mathbf{z}}_1^T\right) \mathbf{r}_k$ to initialize the Krylov subspace. The details of this method can be found in Algorithm \ref{alg:CFIRKS_A} for an Arnoldi type method, assuming that the system matrix is square. Note that, in the first iterations before any restart, we update a flexible Arnoldi relation of type \eqref{eq:flexArnoldi}; but, after the first restart we obtain an augmented flexible Arnoldi relationship of the form:
\begin{equation} \label{eq:augmentedflexArnoldi}
    \mathbf{A}\bar{\mathbf{Z}}_k = \bar{\mathbf{V}}_{k+1}\bar{\mathbf{H}}_{k+1},
\end{equation}
where $\bar{\mathbf{z}}_1 = \mathbf{x}_{k-1}/\|\mathbf{x}_{k-1}\|$, $\bar{\mathbf{V}}_{k+1}$ has orthonormal columns, and $\bar{\mathbf{H}}_{k+1}$ is upper Hessenberg but has the particularity that the element in the second row and first column is 0. Once we have the flexible Arnoldi relation \eqref{eq:flexArnoldi} or a corrected flexible Arnoldi relation \eqref{eq:augmentedflexArnoldi} we can find an approximate correction to the solution by projecting the corrected problem associated to the relevant iteration of the IR scheme applied to the IRN problem. This corresponds to line \ref{alg_step_CFIRKS_A} of Algorithm \ref{alg:CFIRKS_A}; which can be done efficiently by generalizing \st{generalising} the projection \eqref{eq:update} to the relevant \revmsl{corrected} subspaces.

Analogously, an augmented flexible Golub-Kahan type method is detailed in Algorithm \ref{alg:CFIRKS_GKB}, so that, after at least one restart, the following augmented flexible Golub-Kahan relationship is updated at each iteration:
\begin{equation}\label{eq:augmentedflexGKB}
 \mathbf{A}\bar{\mathbf{Z}}_k = \bar{\mathbf{U}}_{k+1} \bar{\mathbf{M}}_{k+1} \quad \text{and} \quad \mathbf{A}^T \bar{\mathbf{U}}_{k+1}=\bar{\mathbf{V}}_{k+1} \bar{\mathbf{S}}_{k+1}.
\end{equation}
Here, the columns of $\bar{\mathbf{Z}}_k=\left[ \mathbf{x}_{k-1}/\|\mathbf{x}_{k-1}\|, {(\mathbf{W}_1 \mathbf{L})^{-1} \bar{\mathbf{v}}_1,..., (\mathbf{W}_k \mathbf{L})^{-1} \bar{\mathbf{v}}_k}\right] \in \mathbb{R}^{n \times k}$ span the solution search space. Moreover, $\bar{\mathbf{M}}_{k+1} \in \mathbb{R}^{ (k+1) \times k}$ is upper Hessenberg, $\bar{\mathbf{U}}_{k+1} \in \mathbb{R}^{m \times (k+1)}$ has orthonormal columns \revision{whose first column corresponds to the normalized} $\left(\mathbf{I}-\bar{\mathbf{z}}_1\bar{\mathbf{z}}_1^T\right) \mathbf{r}_k$, $\bar{\mathbf{S}}_{k+1} \in \mathbb{R}^{(k+1) \times (k+1)}$ is upper triangular, and $\bar{\mathbf{V}}_{k+1} \in \mathbb{R}^{n \times (k+1)}$ has orthonormal columns. Additionally, analogously to CIR-FGMRES, once we have a flexible Golub-Kahan relation \eqref{eq:partialFGKB} or a corrected flexible Golub-Kahan relation \eqref{eq:augmentedflexGKB}, we can find an approximate correction to the solution by projecting the corrected problem associated to the relevant iteration of the IR scheme applied to the IRN problem. This corresponds to line \ref{alg_step_CFIRKS_GKB} of Algorithm \ref{alg:CFIRKS_GKB} which can be done efficiently.

\begin{algorithm}[!ht]
\caption{Corrected Iteratively Refined Flexible GMRES (CIR-FGMRES)}\label{alg:CFIRKS_A}
\text{\bf{Input:}} $\mathbf{A}$, $\mathbf{b}$, $\mathbf{x}_0$, maximum number of iterations $\mathrm{kmax}$ \\
\text{\bf{Output:}} Approximate solution $\mathbf{x}_{\mathrm{kmax}}$ \\
Initialize weights $\mathbf{W}_1= \mathbf{I}$\\
$\mathbf{r}_0 = \mathbf{b} - \mathbf{A} \mathbf{x}_{0}$ \\
$\bar{\mathbf{V}}_1 =\mathbf{r}_0 / \|\mathbf{r}_0 \|$, $\bar{\mathbf{Z}}_{0} = []$ and  $\bar{\mathbf{H}}_{0} =[];$  \\
i = 1\\
\For{$k=1,\dots, \mathrm{kmax}$}{
    \If{inner (restart) stopping criterion}{
        $\bar{\mathbf{Z}}_1 =\bfx_{k-1} / \|\bfx_{k-1}\|_2$ \label{initialZ_CIR-FGMRES} \\
        $\mathbf{q} = \mathbf{A}\bar{\mathbf{Z}}_1 $\\
        $[\bar{\mathbf{H}}_1]_{1,1}= \|\mathbf{q} \|_2$ and $\bar{\mathbf{v}}_1=\mathbf{q}/[\bar{\mathbf{H}}_1]_{1,1}$ \\
        $\bar{\mathbf{V}}_2=[\bar{\mathbf{v}}_1,(\mathbf{I}-\bar{\mathbf{v}}_1 \bar{\mathbf{v}}_1^T) \mathbf{r}_k/\|(\mathbf{I}-\bar{\mathbf{v}}_1 \bar{\mathbf{v}}_1^T) \mathbf{r}_k\|_2]$ \\
        $i = 2$ \\
    }
    Compute [$\bar{\mathbf{z}}_i$, $\bar{\mathbf{h}}_{i}$, $\bar{\mathbf{v}}_{i+1}$] = step\_FA($\mathbf{A}, \bar{\mathbf{V}}_{i}, \mathbf{W}_{k} \mathbf{L}$) \Comment{Algorithm \ref{flexibleArnoldi_step}} \\
    Update $\bar{\mathbf{Z}}_i =[\bar{\mathbf{Z}}_{i-1},\bar{\mathbf{z}}_{i},], \, \bar{\mathbf{H}}_i=\begin{bmatrix} \bar{\mathbf{H}}_{i-1} & \bar{\mathbf{h}}_{i}\\ \mathbf{0} & \phantom{\,} \end{bmatrix},$ and $\bar{\mathbf{V}}_{i+1} =[\bar{\mathbf{V}}_{i}, \bar{\mathbf{v}}_{i+1}]$. \\
    $\displaystyle \mathbf{y}_i = \argmin_{\mathbf{y \in \mathbb{R}^{i}}} \left\{ \left\lVert\begin{bmatrix} \mathbf{A} \mathbf{Z}_i \\ \sqrt{\lambda_i} \mathbf{W}_k \mathbf{L} \mathbf{Z}_i \end{bmatrix} \mathbf{y} - \begin{bmatrix} \mathbf{r}_k \\ - \sqrt{\lambda_i} \mathbf{W}_k \mathbf{L} \, \mathbf{x}_{k-1} \end{bmatrix}\right\rVert_2 \right\}$ \Comment{compute update} \label{alg_step_CFIRKS_A} \\
    $\mathbf{x}_k = \mathbf{x}_{k-1} +\mathbf{Z}_i \mathbf{y}_i$ \\
    $\mathbf{W}_{k+1}=\widetilde{W}(\bfL\mathbf{x}_k)$ according to \eqref{eq:l1weightssmooth} \Comment{update weights} \\
    $\mathbf{r}_{k+1} = \mathbf{b} - \mathbf{A} \mathbf{x}_{k}$ \\
    \If{outer stopping criterion}{
        $\mathrm{kmax} = k$ \\ 
        $\mathrm{break}$
    }
    $i = i+1$
    }
\end{algorithm}

\begin{algorithm}[!ht]
\caption{Corrected Iteratively Refined Flexible LSQR (CIR-FLSQR)}\label{alg:CFIRKS_GKB}
\text{\bf{Input:}} $\mathbf{A}$, $\mathbf{b}$, $\mathbf{x}_0$, maximum number of iterations $\mathrm{kmax}$ \\
\text{\bf{Output:}} Approximate solution $\mathbf{x}_{\mathrm{kmax}}$ \\
Initialize weights $\mathbf{W}_1= \mathbf{I}$\\
$\mathbf{r}_0 = \mathbf{b} - \mathbf{A} \mathbf{x}_{0}$ \\
$\bar{\mathbf{U}}_1 =\mathbf{r}_0 / \|\mathbf{r}_0 \|$, $\bar{\mathbf{Z}}_{0} = []$ and  $\bar{\mathbf{H}}_{0} =[];$  \\
i = 1\\
\For{$k=1,\dots, \mathrm{kmax}$}{
    \If{inner (restart) stopping criterion}{
        $\bar{\mathbf{Z}}_1 =\bfx_{k-1} / \|\bfx_{k-1}\|_2$ \label{initialZ_CFIRKS} \\
        $\mathbf{q} = \mathbf{A}\bar{\mathbf{Z}}_1 $\\
        $[\bar{\mathbf{H}}_1]_{1,1}= \|\mathbf{q} \|_2$ and $\bar{\mathbf{u}}_1=\mathbf{q}/[\bar{\mathbf{H}}_1]_{1,1}$ \\
        $\bar{\mathbf{U}}_2=[\bar{\mathbf{u}}_1,(\mathbf{I}-\bar{\mathbf{u}}_1 \bar{\mathbf{u}}_1^T) \mathbf{r}_k/\|(\mathbf{I}-\bar{\mathbf{u}}_1 \bar{\mathbf{u}}_1^T) \mathbf{r}_k\|_2]$ \\
        $i = 2$ \\
    }
    Compute [$\mathbf{z}_i, \mathbf{s}_{i}, \mathbf{v}_{\revmsl{i}}, \mathbf{m}_{\revmsl{i}}, \mathbf{u}_{i+1}$] = step\_FGK($\mathbf{A}, \bar{\mathbf{U}}_{i}, \bar{\mathbf{V}}_{i-1}, \mathbf{W}_{k} \mathbf{L}$) \Comment{Algorithm \ref{flexibleGKB_step}}  \\
    Update $\mathbf{Z}_i=[\mathbf{Z}_{i-1}, \mathbf{z}_{i}],$ \,$\mathbf{S}_i=\begin{bmatrix} \mathbf{S}_{i-1} & \mathbf{s}_{i}\\ \mathbf{0} & \phantom{\,} \end{bmatrix},$  \, $\mathbf{V}_{\revmsl{i}}=[\mathbf{V}_{\revmsl{i-1}}, \mathbf{v}_{\revmsl{i}}]$  \\
     Update $\mathbf{U}_{i+1}=[\mathbf{U}_{i}, \mathbf{u}_{i+1}], \,\mathbf{M}_{\revmsl{i}}=\begin{bmatrix} \mathbf{M}_{\revmsl{i-1}} & \mathbf{m}_{\revmsl{i}}\\ \mathbf{0} & \phantom{\,} \end{bmatrix}$ \\
    $\displaystyle \mathbf{y}_i = \argmin_{\mathbf{y \in \mathbb{R}^{i}}} \left\{ \left\lVert\begin{bmatrix} \mathbf{A} \mathbf{Z}_i \\ \sqrt{\lambda_i} \mathbf{W}_k \mathbf{L} \mathbf{Z}_i \end{bmatrix} \mathbf{y} - \begin{bmatrix} \mathbf{r}_k \\ - \sqrt{\lambda_i} \mathbf{W}_k \mathbf{L} \, \mathbf{x}_{k-1} \end{bmatrix}\right\rVert_2 \right\}$ \Comment{compute update} \label{alg_step_CFIRKS_GKB} \\
    $\mathbf{x}_k = \mathbf{x}_{k-1} +\mathbf{Z}_i \mathbf{y}_i$ \\
    $\mathbf{W}_{k+1}=\widetilde{W}(\bfL\mathbf{x}_k)$ according to \eqref{eq:l1weightssmooth} \Comment{update weights} \\
    $\mathbf{r}_{k+1} = \mathbf{b} - \mathbf{A} \mathbf{x}_{k}$ \\
    \If{outer stopping criterion}{
        $\mathrm{kmax} = k$ \\ 
        $\mathrm{break}$
    }
    $i = i+1$
    }
\end{algorithm}

\subsection{Criteria for the choice of regularization parameters}\label{sec:reg_parag}

Many algorithms presented \revision{or referenced in this work} require the selection of a regularization parameter(s). \revision{Herein, we utilize the discrepancy principle to determine the regularization parameters under the assumption that a good estimate for the noise level ($\mathrm{nl}$) corrupting the linear system is available. This is defined as} \st{In this paper, we opt for using the discrepancy principle. Assuming that a good estimate for the noise level ($\mathrm{nl}$) is available, which defined as}  
\begin{equation}
\mathrm{nl} = \frac{\|\bfe\|_2}{\|\bfb - \bfe\|_2}
\end{equation}
where we select $\lambda_k$ such that
\begin{equation}\label{DP}
 \revmsl{\|\bfr_{k+1}(\lambda_k)\|_2 / \|\bfb \|_2 =}\|\bfA \bfx_{\revmsl{k}}(\lambda_k)-\bfb\|_2 / \|\bfb \|_2 = \mathrm{nl}\cdot\tau
\end{equation}
where $\tau$ is a safety parameter close to 1 that \revision{is used to avoid over-regularization when $\tau > 1$.} In the numerical experiments, we consider $\tau=1$. \revision{We clarify that $\bfx_k(\lambda_k)$ is defined to be the $k^{th}$ solution as a function of the regularization parameter which can change at each iteration.} Moreover, for any projected problem, we can write  $\bfx_k(\lambda_k) = \bfZ_k \revmsl{\bfy}_k(\lambda_k) + \bfx_{k-1}$; and 
$$\revmsl{\bfr_{k+1}}(\lambda_k)=\bfA \bfx_k(\lambda_k)-\bfb = \bfA \bfZ_k \revmsl{\bfy}_k(\lambda_k) + \bfA\bfx_{k-1}-\bfb= \bfA \bfZ_k \revmsl{\bfy}_k(\lambda_k) - \bfr_k;$$
so that the residual norm in equation \eqref{DP} can be computed efficiently:
$$\|\bfr_{\revmsl{k+1}}(\lambda_k)\|_2= \| \bfA \bfZ_k \revmsl{\bfy}_k(\lambda_k) - \bfr_k\|_2 = \| \bfT_k\revmsl{\bfy}_k(\lambda_k) - \bfX_k^T\bfr_k\|_2.$$
Notation-wise, $\bfT_k = \bfH_{\revmsl{k+1}}$ and $\bfX_k=\bfV_k$ if we are using IR-FGMRES; which involves flexible Arnoldi, or the barred counterparts $\bar{\bfH}_k$ and $\bar{\bfV}_k$ for CIR-FGMRES. \revmsl{We note, with a small abuse of notation, that in the latter case we would also have} \st{(note that with a small notation abuse, in this case we would also have} $\bar{\bfZ}_k$ instead of $\bfZ_k$. \revision{On the other hand, we have $\bfT_k = \bfM_{\revmsl{k+1}}$ and $\bfX_k=\bfU_k$ if we are using IR-FLSQR, the flexible Golub-Kahan relation \eqref{eq:partialFGKB}, or the bared counterparts if we are using CIR-FLSQR \eqref{eq:augmentedflexGKB}.}

\revision{We note that if a solution for \eqref{DP} does not exist, then there is not a solution in the constructed search space with a residual norm smaller than the noise level; so we take $\lambda_k=0$, which will minimize the difference between the right and left-hand sides of equation \eqref{DP}.}

In our numerical examples in Section \ref{sec:numerics}, we also show results for the optimal regularization parameter, which can be computed by
\begin{equation}\label{opt}
 \lambda_k=\arg\min_\lambda \| \bfx_{\revmsl{k}}(\lambda_k)-\mathbf{x_{\textbf{true}}}\|_2 
\end{equation}
Of course, this is not a valid parameter choice in practice since it requires the knowledge of the true solution $\mathbf{x_{\textbf{true}}}$; however, this serves as a good indicator of how well the method can perform and gives a way to assess the quality of other regularization parameter choices.
 
We are also interested in using automatic criteria to restart the search spaces. The choice of the stopping criteria (to trigger a restart) is inherently linked to the choice of regularization parameter. For example, note that since we are using the discrepancy principle to choose the regularization parameter, the residual norm will always take the value of the noise level (or occasionally a higher value if this is not attainable), so the discrepancy principle cannot be used to stop the iterations. In this paper, we choose when to restart the subspaces given the stabilization of the regularization parameter. That means, we restart the iterations when
\begin{equation}\label{stop_crit}
    \frac{|\lambda_{k}-\lambda_{k-1}|}{\lambda_{k-1}} \leq \mathrm{tol} \quad \text{and} \quad 
    \frac{|\lambda_{k-1}-\lambda_{k-2}|}{\lambda_{k-2}} \leq \mathrm{tol};
\end{equation}
for a user given tolerance $\mathrm{tol}$. Note that since $\lambda_{k}$ is not a monotonic quantity, we use two consecutive differences to assess stabilization. This is consistent with the IRTools toolbox codes \cite{IRtools} which are used as a comparison to evaluate the performance of the newly proposed methods.

\revision{We highlight that since we are dealing with projection methods that often require a relatively small subspace dimension to find a meaningful solution, other parameter choice strategies can also be used, see, e.g.,} \st{Note that; since we are dealing with projection methods; the projected problem at each iteration has a very small dimension; so other parameter choices can also be used, see, e.g.} \cite{Chung2024survey, Gazzola2020Parameters}. However, for some parameter choice methods such as the generalized cross-validation (GCV), some care needs to be taken since the regularization term has a right-hand side.  
Additionally, other options for the combination of regularization and stopping parameter choices are also possible; for example, using GCV or WGCV for the regularization parameter and GCV as the stopping criterion, \cite{Chung2008wgcv}; or secant updates on the DP for the regularization parameter and the DP as the stopping criterion \cite{GAZZOLA2014180, Gazzola2020Parameters}. \revision{We leave these avenues open for future research.}


\section{Theoretical considerations} \label{sec:theory}

In this section we give an overview on what can be said on the convergence of our proposed methods: IR-FGMRES, IR-FLSQR, CIR-FGMRES and CIR-FLSQR.
In all the derivations, we consider that the regularization parameter $\lambda \geq 0$ is fixed throughout the iterations\revision{. That is, we drop} the dependency on the iteration count with respect to the more general notation $\lambda_k$ that we use in the \revision{prior sections.} \st{methods section.}

First, note that IR-FGMRES and IR-FLSQR without restarts solve 
\begin{eqnarray}
     \bfx_k &=& \bfx_{k-1}+\bfh_k \nonumber\\
     &=& \bfx_{k-1}+\arg\min_{\bfh \in \mathcal{R}(\bfZ_k)}\|\bfA\bfh-\bfr_{k-1}\|_2^2+\lambda \|\widetilde{W}(\bfL \bfx_{k-1})\bfL(\bfh+\bfx_{k-1})\|_2^2.
     \label{eq:proofs1}
\end{eqnarray}
Since $\bfx_{k-1}\in\mathcal{R}(\bfZ_{k-1}) \subset \mathcal{R}(\bfZ_{k})$, and using $\bfr_{k-1}=\bfb-\bfA\bfx_{k-1}$; the solution $\bfx_k$ in \eqref{eq:proofs1} solves
\begin{eqnarray}
    \min_{\bfx \in \mathcal{R}(\bfZ_k)}\|\bfA\bfx-\bfb\|_2^2+\lambda \|\widetilde{W}(\bfL\bfx_{k-1})\bfL\bfx\|_2^2.
     \label{eq:proofs2}
\end{eqnarray}
This means that, for a fixed regularization parameter choice and without restarts, IR-FGMRES and IR-FLSQR are mathematically equivalent to IRW-FGMRES and IRW-FLSQR.

Now we consider IR-FGMRES and IR-FLSQR with restarts. The solution at iteration $k$, where $i$ iterations have happened after a restart (for $i\leq k$), given by either method, can be written as 
\begin{eqnarray}
      \bfx_{k}= \bfx_{k-1}+\arg\min_{\bfh \in \mathcal{R}(\bfZ_i)}\|\bfA(\bfh+\bfx_{k-1})-\bfb\|_2^2+\lambda \|\widetilde{W}(\bfL\bfx_{k-1})\bfL(\bfh+\bfx_{k-1})\|_2^2
     \label{eq:proofs3}
\end{eqnarray}
where $\bfZ_i$ is defined according to \eqref{eq:flexArnoldi} for IR-FGMRES or \eqref{eq:partialFGKB} for IR-FLSQR. Analogously; for CIR-FGMRES and CIR-FLSQR we have an analogous expression to \eqref{eq:proofs3}, but where $\bfh \in \mathcal{R}(\bar{\bfZ}_i)$ and $\bar{\bfZ}_i$ are defined in \eqref{eq:augmentedflexArnoldi} for CIR-FGMRES and  \eqref{eq:augmentedflexGKB} for CIR-FLSQR.  Moreover, in this case, note that the minimization in the analogous equation to \eqref{eq:proofs3} is equivalent to  
\begin{eqnarray}
      \min_{\bfx \in \mathcal{R}(\bar{\bfZ}_i)}\|\bfA\bfx-\bfb\|_2^2+\lambda \|\widetilde{W}(\bfL\bfx_{k-1})\bfL\bfx\|_2^2,
     \label{eq:proofs4}
\end{eqnarray}
since $\bfx_{k-1}\in \mathcal{R}(\bar{\bfZ}_i)$. \\

Now, for $\lambda \geq 0$, we define the following functionals:
\begin{eqnarray}
    T(\bfx) &=& \|\bfA\bfx-\bfb\|_2^2+\lambda \|\widetilde{W}(\bfL\bfx)\bfL\bfx\|_2^2 \label{eq:T}\\
    T_k(\bfx) &=& \|\bfA\bfx-\bfb\|_2^2+\lambda \|\widetilde{W}(\bfL\bfx_{k-1})\bfL\bfx\|_2^2 =\|\bfA\bfx-\bfb\|_2^2+\lambda \|\bfW_{k}\bfL\bfx\|_2^2;\label{eq:Tk}
\end{eqnarray}
where $T(\bfx)$ is the functional that we want to minimize and the $T_k(\bfx)$ are the corresponding quadratic majorants of $T(\bfx)$ at $\bfx=\bfx_{k-1}$ that we consider at each iteration.
\revision{Here, we point out that} $T(\bfx_{k-1}) = T_k(\bfx_{k-1})$.

\begin{proposition}\label{prop1}
    Let $\{\bfx_k\}_k$ be the sequence of approximate solutions computed by either IR-FGMRES, IF-FLSQR, CIR-FGMRES, or CIR-FLSQR (in Algorithms \ref{alg:FIRKS}, \ref{alg:CFIRKS_A} and \ref{alg:CFIRKS_GKB}, respectively). Assume that, for all $\bfW_k=\widetilde{W}(\bfL\bfx_k)$ the following holds: $\mathcal{N}(\bfA)\cap \mathcal{N}(\bfW_k \bfL)$. Then, given a fixed $\lambda \geq 0$, for all $k$: $$0\leq T(\bfx_k)\leq T(\bfx_{k-1}).$$
\end{proposition}
\begin{proof}
First note that the functional $T(\bfx)$ in \eqref{eq:T} is non-negative, so we only need to prove the second inequality. For all methods, it can be easily seen that:
\begin{equation}\label{proof1}
    T(\bfx_k) \leq T_k(\bfx_k) \leq T_k(\bfx_{k-1}) = T(\bfx_{k-1}).
\end{equation}
Here, the first inequality holds since $T_k(\bfx)$ is a majorant of $T(\bfx)$ and the last equality holds since $T_k(\bfx_{k-1}) =T(\bfx_{k-1})$ by definition; this can be verified by evaluating \eqref{eq:T} and \eqref{eq:Tk} at $\bfx_{k-1}$. Therefore, one only needs to check the second inequality for all the methods listed in the Proposition. Note that $\bfx_k$ is computed at each iteration such that 
$$T(\bfx_k)=\min_{\bfh} \|\bfA(\bfh+\bfx_{k-1})-\bfb\|_2^2+\lambda \|\widetilde{W}(\bfL\bfx_{k-1})\bfL(\bfh+\bfx_{k-1})\|_2^2,$$
where either $\bfh \in \mathcal{R}(\bfZ_i)$ (for IR-FGMRES and IR-FLSQR), or $\bfh \in \mathcal{R}(\widetilde{\bfZ}_i)$ (for CIR-FGMRES and CIR-FLSQR). \revision{However, by taking} \st{But, taking} $\bfh=0$, one obtains $T_k(\bfx_{k-1})$, so that $T_k(\bfx_{k})\leq T_k(\bfx_{k-1})$.
\end{proof}

\begin{corollary}
Under the same assumptions of Proposition \ref{prop1}, $\{T(\bfx_k)\}_{k\geq 1}$ has a stationary point.
\end{corollary}

\begin{proof}
The sequence $\{T(\bfx_k)\}_{k\geq 1}$ is non-negative and monotonically decreasing.
\end{proof}

\begin{proposition}
Under the same assumptions of Proposition \ref{prop1}, the sequence of approximate solutions $\{x_k\}_{k\geq 1}$ computed after $k$ steps of either IR-FGMRES or IR-FLSQR without restarts is such that
$$\lim_{k\rightarrow \infty} \|x_k-x_{k-1}\|_2 = 0.$$
Moreover, the sequence converges to the unique solution of $T(\bfx)$ in \eqref{eq:T}. 
\end{proposition}

\begin{proof}
    Since IR-FGMRES and IR-FLSQR are mathematically equivalent to IRW-FGMRES and IRW-FLSQR, the theoretical guarantees for convergence follow readily from \cite[Theorem 3.4]{Gazzola2021IRW}. 
\end{proof}

\begin{proposition}
Under the same assumptions of Proposition \ref{prop1}, the sequence of approximate solutions $\{x_k\}_{k\geq 1}$ computed after $k$ steps of either IR-FGMRES, IR-FLSQR, CIR-FGMRES or CIR-FLSQR after any possible number of restarts has a convergent subsequence.
\end{proposition}

\begin{proof}
By the Bolzano–Weierstrass theorem, we only need to prove that all the elements of the sequence $\{x_k\}_{k\geq 1}$ have bounded norm. First, note that, since $\lambda \geq 0$,
$$\|\bfA \bfx_k - \bfb\|^2_2 \leq T(\bfx_k)\leq T(\bfx_1) := c_1, $$
where the second inequality holds due to Proposition \ref{prop1}. Moreover, due to inequality \eqref{proof1},
$$\|\bfW_k\bfL \bfx_k\|^2_2 \leq \lambda^{-1} T_k(\bfx_{k}) \leq \lambda^{-1} T(\bfx_{k-1})\leq \lambda^{-1} c_1,$$
so that 
$$\|\bfA \bfx_k - \bfb\|^2_2 + \|\bfW_k\bfL \bfx_k\|^2_2 \leq (1+\lambda^{-1})c_1, $$
and since $\mathcal{N}(\bfA)\cap \mathcal{N}(\bfW_k \bfL)$, there exists a constant $C$ such that
$$\|\bfx_k\|_2^2\leq C \quad \forall k\geq 1.$$
\end{proof}
Note that this proof is similar to that of \cite{Buccini2023restart}.


\section{Numerical experiments}\label{sec:numerics}
In this section, we show three different numerical experiments to highlight the performance of our newly proposed methods. First, we show a small one-dimensional deblurring example with the intention to showcase the importance of using \revision{an effective} search subspace for the solution. \revision{We note that this is a very small problem, so restarting is unnecessary.} Then, we present two imaging examples: a deblurring example and an oversampled, but highly noisy tomography problem. In these examples, we show the benefits of restarting both in terms of memory and convergence.

In all examples, comparisons with the following methods are performed:
\begin{itemize}
    \item H-GMRES / H-LSQR: hybrid versions of GMRES and LSQR, see, e.g. \cite{Chung2024survey}. Implementation provided in \cite{IRtools}.
    \item H-FGMRES / H-LSQR: hybrid versions of flexible GMRES and LSQR, see, e.g. \cite{JulianneSilvia}. Implementation provided in \cite{IRtools}.
    \item IRW-FGMRES / IRW-LSQR: iteratively reweighted flexible GMRES and LSQR, as described in \cite{Gazzola2021IRW}.
    \item IRN-GMRES / IRN-LSQR: inner-outer scheme where the weights are fixed thoughout the inner iterations, and the solution is computed using \revision{H-GMRES} or \revision{H-LSQR}.
    \item FISTA: Fast Iterative Shrinkage-Thresholding Algorithm \revision{as described in \cite{beck2009fast}} with user-specified regularization parameters. Implementation provided in \cite{IRtools}.
    \item SpaRSA: Sparse Reconstruction by Separable Approximation, as described in \cite{wright2009sparse} with user-specified regularization parameters, and using their implementation.
\end{itemize}

Unless stated otherwise, all regularization parameters are chosen using the discrepancy principle; and the stopping criteria to trigger a restart is the stabilization of the regularization parameter \revision{(as described in Section \ref{sec:reg_parag})} or a prescribed maximum amount of basis vectors. All the results presented in this section were performed using MATLAB on an M-series MacBook Pro.

\subsection{One dimensional deblurring problem}
The first numerical experiment concerns a one dimensional deblurring example, where the entries of the system matrix $\bfA\in \mathbb{R}^{64 \times 64}$ are defined by 
$$[\bfA]_{i,j}=\frac{1}{2\sqrt{2\pi}} e^{-\frac{(i-j)^2}{8}};$$
and, following \cite{Trussell1983Convergence}, the true solution $\mathbf{x_{\textbf{true}}}$ shows 4 peaks simulating an X-ray spectrum. Moreover, the right hand side $\mathbf{b}$ is corrupted by Gaussian noise \revision{with zero mean and fixed variance to correspond to a $1\%$ noise level.} \st{with $1\%$ noise level.} The solution and the measurements can be observed in Figure \ref{Ex1_setting}.

\begin{figure}[ht!]
    \centering
    \begin{subfigure}[b]{0.4\textwidth}
        \centering
        \includegraphics[width=0.7\textwidth]{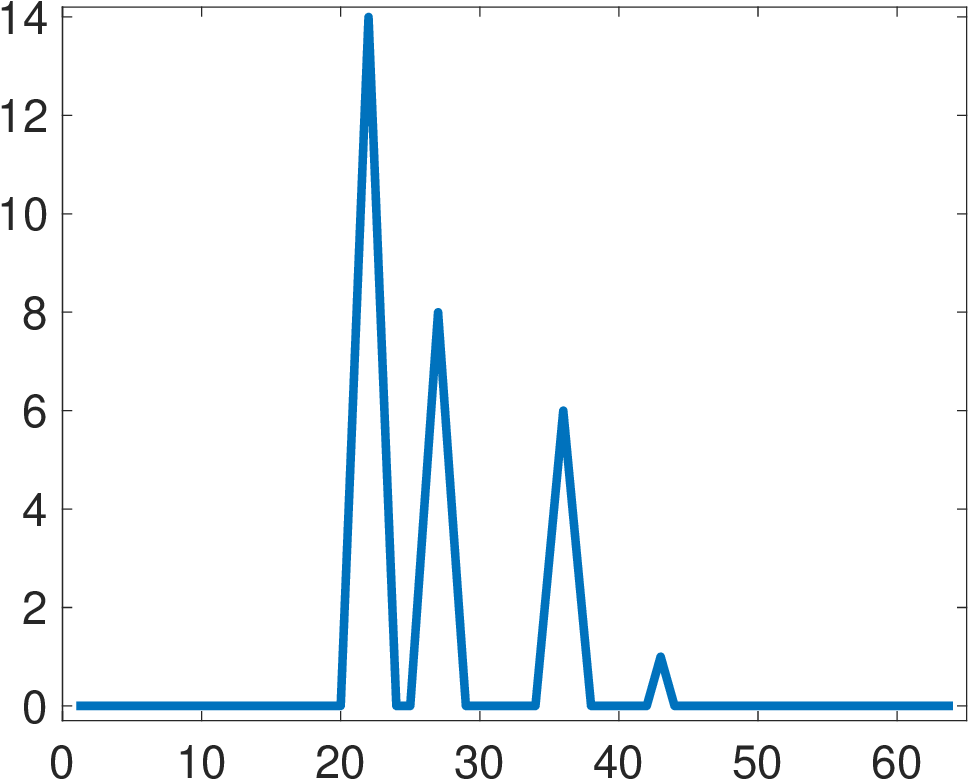}
        \caption*{True solution: $\mathbf{x_{\textbf{true}}}$}
    \end{subfigure}
    \begin{subfigure}[b]{0.4\textwidth}
        \centering
        \includegraphics[width=0.7\textwidth]{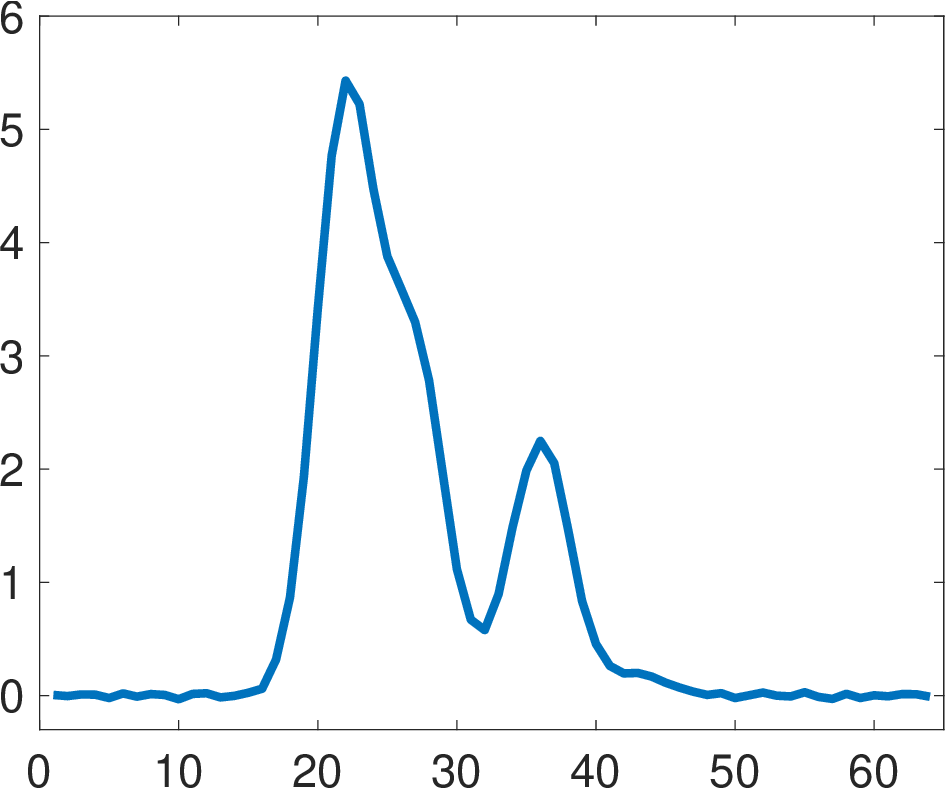}
        \caption*{Noisy measurements: $\mathbf{b}$}
    \end{subfigure}
    \caption{Setting for the 1D signal restoration example Spectra.}\label{Ex1_setting}
\end{figure}

One of the things that we would like to highlight in this example is the suitability of the basis vectors provided by flexible Kryov subspaces with respect to their standard (non-flexible) counterparts. In Figure \ref{Ex1_basis}, we display a few basis vectors \revision{(row-wise) corresponding to the search spaces of our proposed methods:} \st{the search spaces corresponding to the new methods:} IR-FGMRES and IR-FLSQR, as well as those from standard GMRES and LSQR. In all frames, we further show the true solution (in black discontinuous lines). One can easily observe that the flexible bases are much more tailored to the solution; in particular, they are much closer to 0 when the solution is zero\revision{; this is due} to the effect of the sparsity-promoting weights coming from the IRN framework.

\begin{figure}[ht!]
    \centering
    \includegraphics[width=1\textwidth]{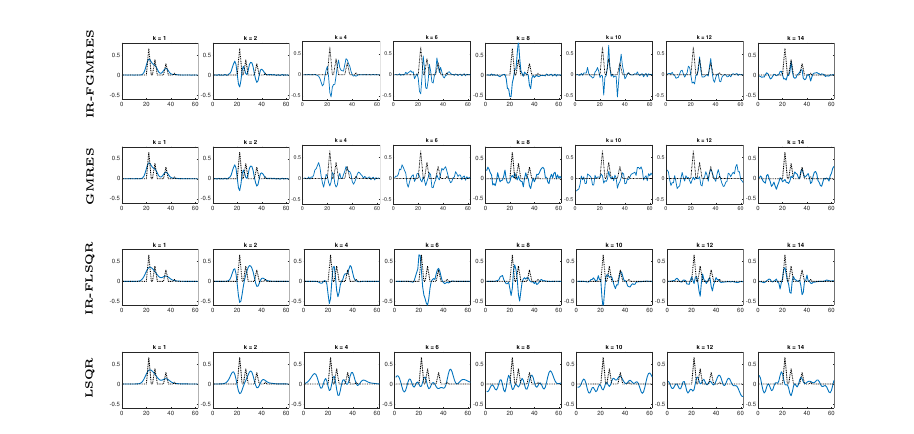}    
    \caption{Selection of basis vectors for the search spaces corresponding to the new methods: IR-FGMRES and IR-FLSQR; in comparison with the standard GMRES and LSQR for the test problem Spectra. Underlaid, in black discontinuos lines, the true solution.}\label{Ex1_basis}
\end{figure}

Moreover, in Figure \ref{Ex1_performance}, we show the performance of the new methods. In the first column, we can see \revision{the good agreement between the errors when comparing the use of the optimal regularization parameter to the one determined by the discrepancy principle, which indicates} that the discrepancy principle is a very good regularization parameter choice method for this problem. Therefore, we present results for this parameter choice in the second and third column. In particular, in the second column, we show that the new methods are very competitive with respect to other standard solvers based on standard and flexible Krylov subspace methods; and in the last column, we show that the new methods display a much faster convergence than other standard solvers considering $\ell_1$ regularization. Moreover, we note that FISTA and SpaRSA require a user-specified regularization parameter \revision{making these methods more cumbersome to use compared with Krylov methods discussed herein}. In this case, we use the regularization parameter at the end of the iterations given by IR-FGMRES (top row) and IR-FLSQR (bottom row). \revision{We again clarify that because this is a very small problem, we do not make use of restarts.}\st{Finally, since this is a very small problem, note that we do not make use of restarts.}

\begin{figure}[ht!]
    \centering
    {\rotatebox[origin=c]{90}{\hspace{3.3cm} \small(flexible) Arnoldi}} 
    \begin{subfigure}[b]{0.30\textwidth}
        \centering
        \includegraphics[width=0.95\textwidth]{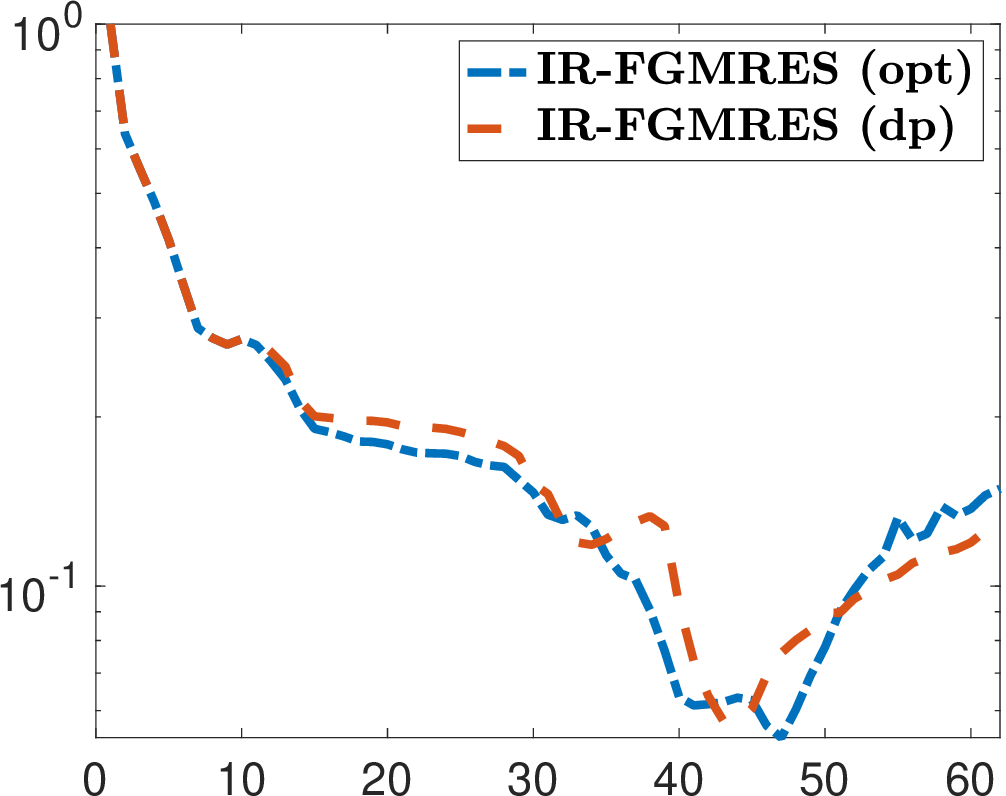}
    \end{subfigure}
    \begin{subfigure}[b]{0.30\textwidth}
        \centering
        \includegraphics[width=0.95\textwidth]{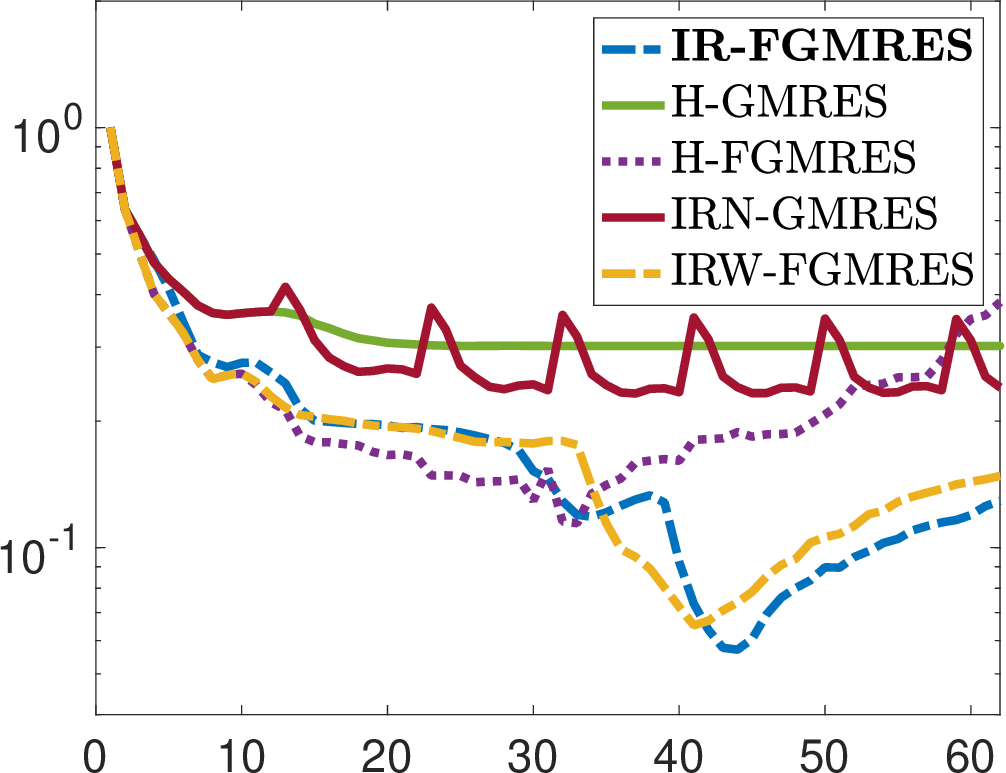}
    \end{subfigure}
    \begin{subfigure}[b]{0.30\textwidth}
        \centering
        \includegraphics[width=0.95\textwidth]{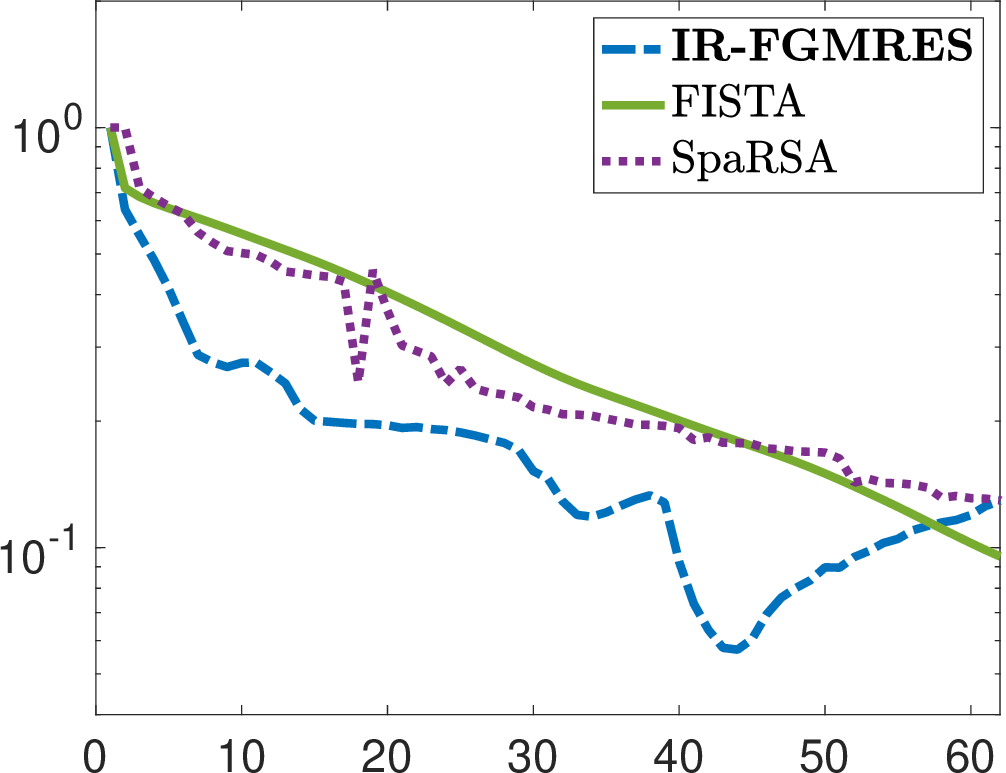}
    \end{subfigure} \\ \vspace{-2.7cm}
        {\rotatebox[origin=c]{90}{\hspace{4.3cm}\small(flexible) Golub-Kahan}}
    \begin{subfigure}[b]{0.30\textwidth}
        \centering
        \includegraphics[width=0.95\textwidth]{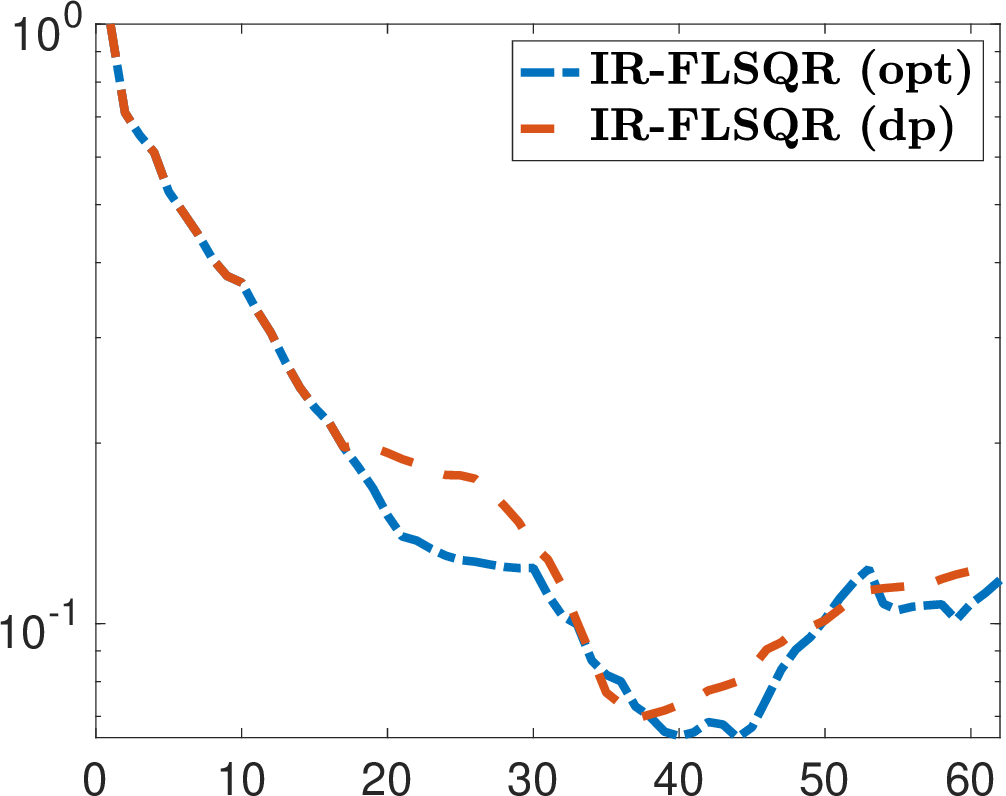}
         \caption*{Different $\lambda_k$ choices}
    \end{subfigure}
    \begin{subfigure}[b]{0.30\textwidth}
        \centering
        \includegraphics[width=0.95\textwidth]{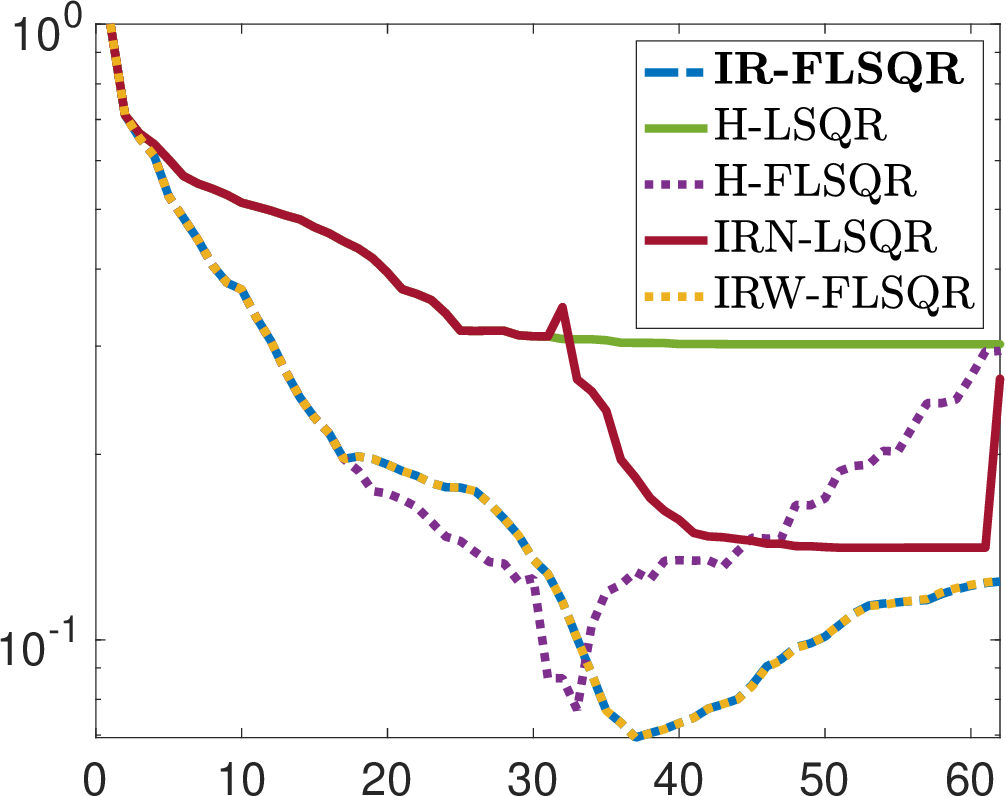}
        \caption*{Other (flexible) Krylov methods}
    \end{subfigure}
    \begin{subfigure}[b]{0.30\textwidth}
        \centering
        \includegraphics[width=0.95\textwidth]{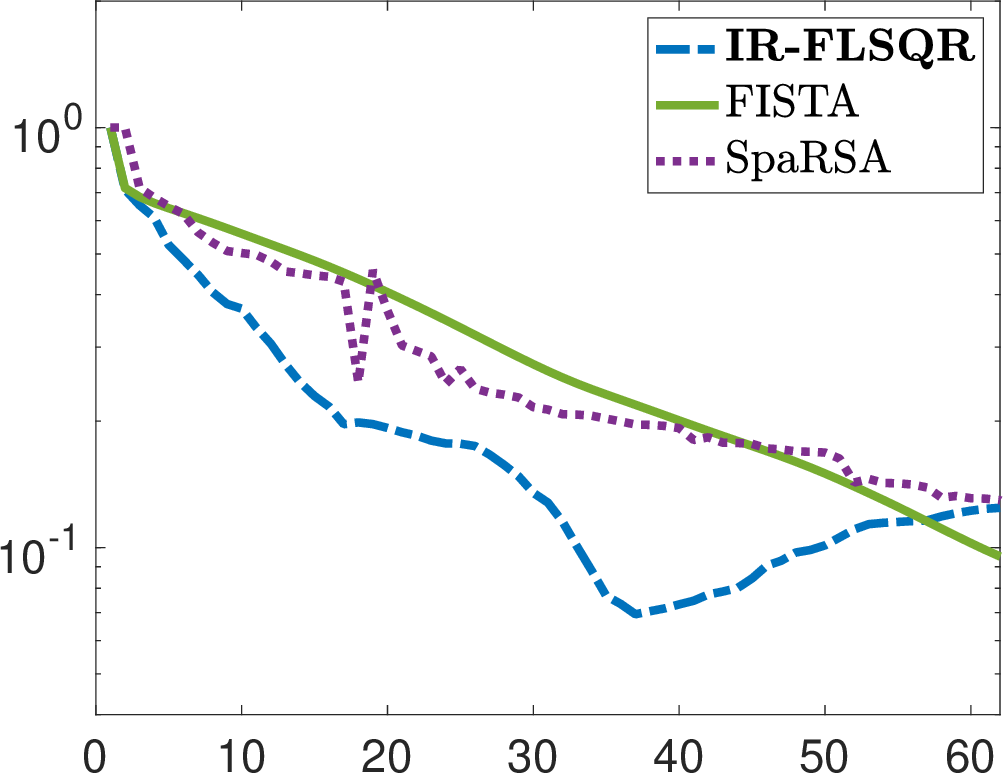}
        \caption*{Other $\ell_1$ solvers}
    \end{subfigure}
    \vspace{-3.5cm}
    \caption{Different error norm histories across the iterations for the test problem Spectra. In the top row, comparisons between methods based on the (flexible) Arnoldi method; on the bottom row, comparisons between methods based on the (flexible) Golub-Kahan method. In the second and third column, the discrepancy principle is used at each iteration when possible. Note that FISTA and SpaRSA require a user-specified regularization parameter.}\label{Ex1_performance}
\end{figure}

\subsection{Deblurring star cluster example}
In this example, we look at a deblurring problem modeled by a square matrix $\mathbf{A} \in \mathbb{R}^{65536 \times 65536}$ representing a spatially variant blur. A detailed explanation of this forward operator as well as the efficient implementation of relevant matrix-vector products can be found in \cite{Nagy1998degraded}. The exact solution is a sparse image (in the sense that only 7.2 $\%$ of the pixels are greater than $10^{-10}$) which models a starry sky, and is taken from \cite{RestoreTools}. To highlight the effect of the regularization term, we again consider Gaussian noise with a high noise level of $50\%$. Note that this example is also featured in \cite{Gazzola2014GAT, Gazzola2021IRW}. The solution and the noisy measurements can be observed in Figure \ref{Ex2_setting}.

\begin{figure}[ht!]
    \centering
    \begin{subfigure}[b]{0.47\textwidth}
        \centering
        \includegraphics[width=0.7\textwidth]{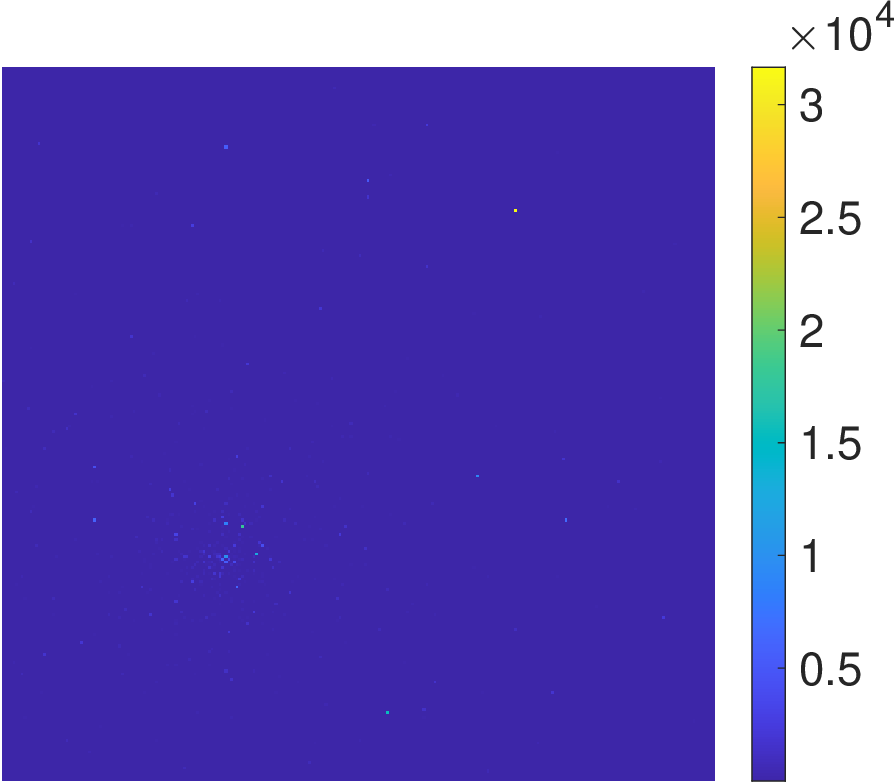}
        \caption*{True solution: $\mathbf{x_{\textbf{true}}}$}
    \end{subfigure}
    \begin{subfigure}[b]{0.47\textwidth}
        \centering
        \includegraphics[width=0.7\textwidth]{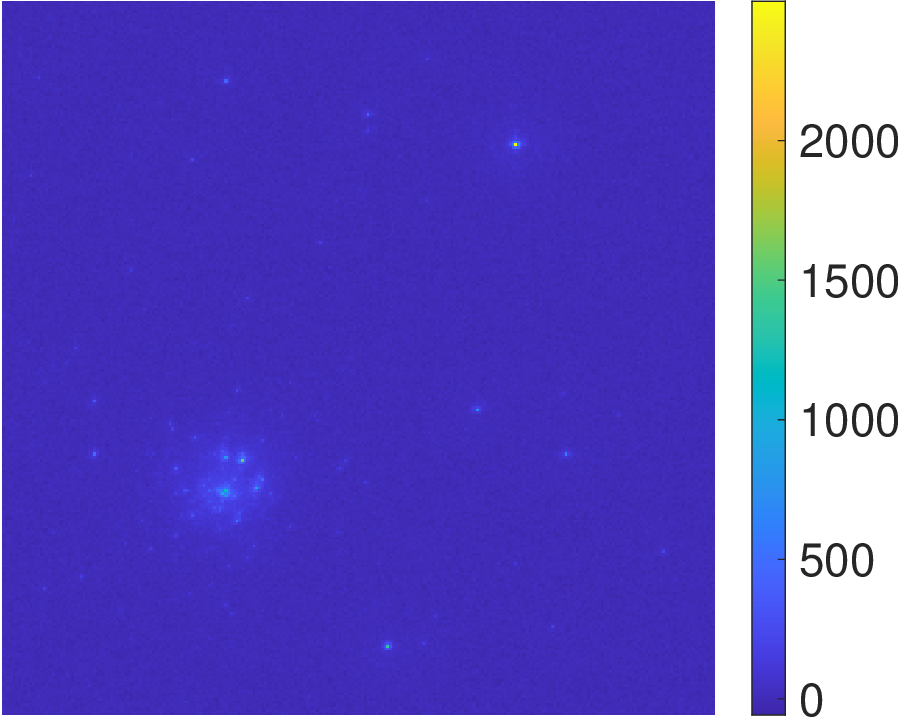}
        \caption*{Noisy measurements: $\mathbf{b}$}
    \end{subfigure}
    \caption{True image and measurements for the test problem star cluster. Note that the color bars are in different scales. This is because the measurements have lost contrast due to the blurring, and it would be hard to visualize $\mathbf{b}$ otherwise.} \label{Ex2_setting}
\end{figure}

In Figure \ref{Ex2_performance} we can observe the relative error norms throughout the iterations. For the new methods IR-FMGRES and CIR-FGMRES (as well as IRN-GMRES), we use (subspace) restarts which are triggered by the stabilization of the regularization parameter, and never store more than 30 basis vectors to \revision{simulate}\st{reproduce} a setting where we have memory constraints or fast memory \revision{is expensive to utilize.}\st{comes at a premium.} In the first panel, we \revision{study the choice of the regularization parameter.} First, \revision{we note that for this particular example and for the (not corrected) iterative refinement method (IR-FGMRES),} \st{note that for the (not corrected) iterative refinement method (IR-FGMRES), and for this example,} the discrepancy principle does a good job at finding a regularization parameter \revision{that results in, on average, a low relative error. However, the behavior is highly oscillatory, which is not ideal. When we consider the corrected iterative refinement method (CIR-FGMRES) however, the solution is much more stable with respect to the choice of the regularization parameter. We do note however, that the average error falls short compared to the use of the optimal regularization parameter choice.} 
\st{However, this is not true for the corrected iterative refinement method (CIR-FGMRES); which shows a more stable behaviour but falls short compared to the optimal choice.} For consistency, we have decided to show the results using the \revision{discrepancy principle}; but further exploration with other regularization parameters might be needed in practice. 

In the second panel of Figure \ref{Ex2_performance}, it can be observed that the performance of the new methods in terms of relative error norm is better than other standard flexible and hybrid Krylov methods. Finally, in the third panel, we can observe that the new corrected method has a much faster convergence than FISTA, and that initially the error decays much faster than for SpaRSA. However, in this example, the error norm stagnates when using the discrepancy principle as a regularization parameter, so we also provide the results using the optimal parameter choice for reference. Note that, in any way, FISTA and SpaRSA require a user-specified regularization parameter; so we use the regularization parameter at the end of the iterations given by CIR-FGMRES. Flexible and hybrid Krylov methods, including IR-FGMRES and CIR-FGMRES, have the natural advantage of allowing for efficient automatic parameter selection throughout the iterations. In this case, we use the regularization parameter at the end of the iterations given by IR-FGMRES (top row) and IR-FLSQR (bottom row).

\begin{figure}[ht!]
    \centering
    \begin{subfigure}[b]{0.32\textwidth}
        \centering
        \includegraphics[width=0.95\textwidth]{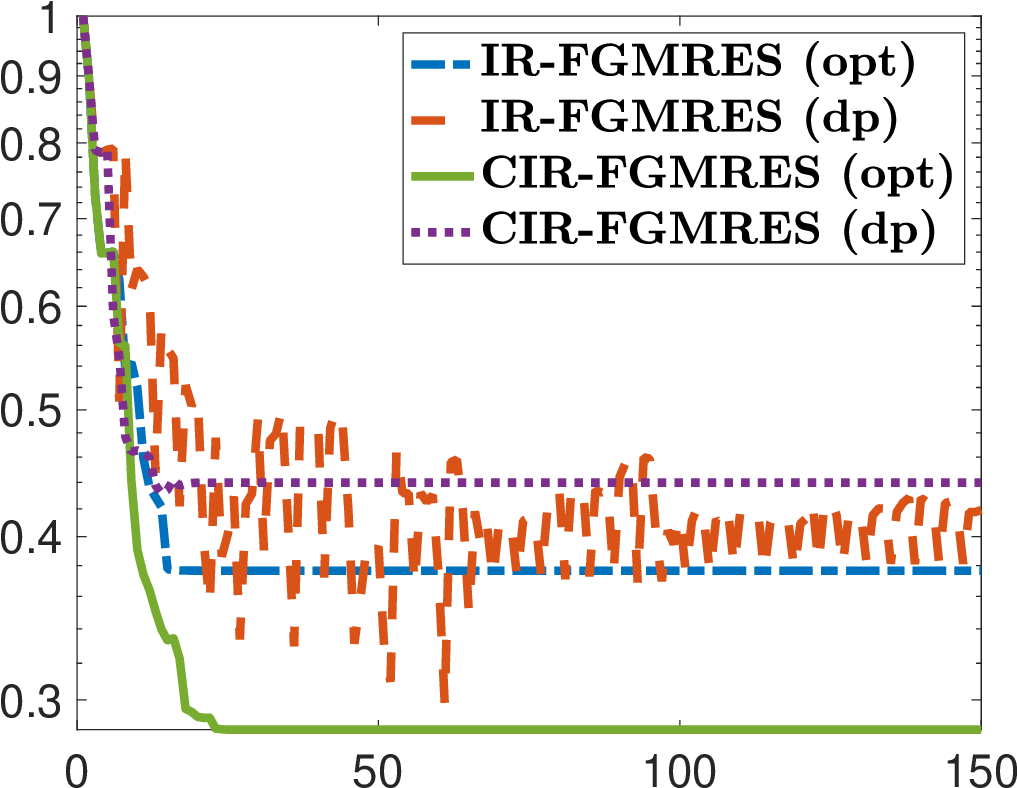}
        \caption*{Difference $\lambda_k$ choices}
    \end{subfigure}
    \begin{subfigure}[b]{0.32\textwidth}
        \centering
        \includegraphics[width=0.95\textwidth]{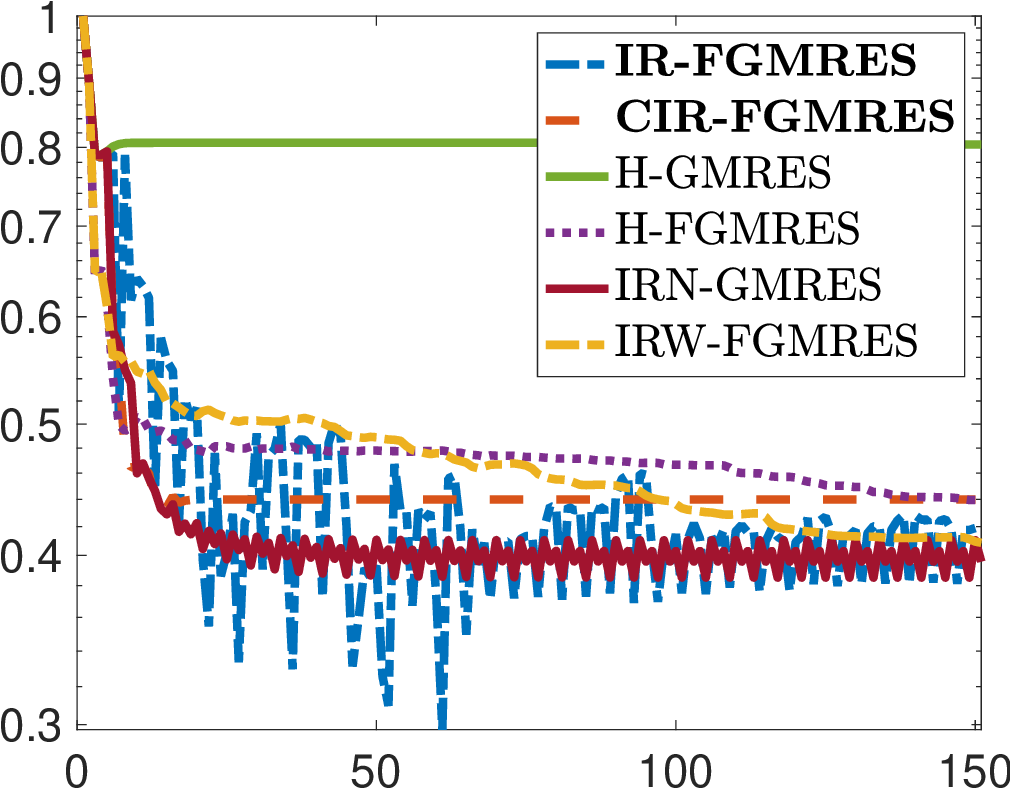}
        \caption*{Other (flexible) Krylov methods}
    \end{subfigure}
    \begin{subfigure}[b]{0.32\textwidth}
        \centering
        \includegraphics[width=0.95\textwidth]{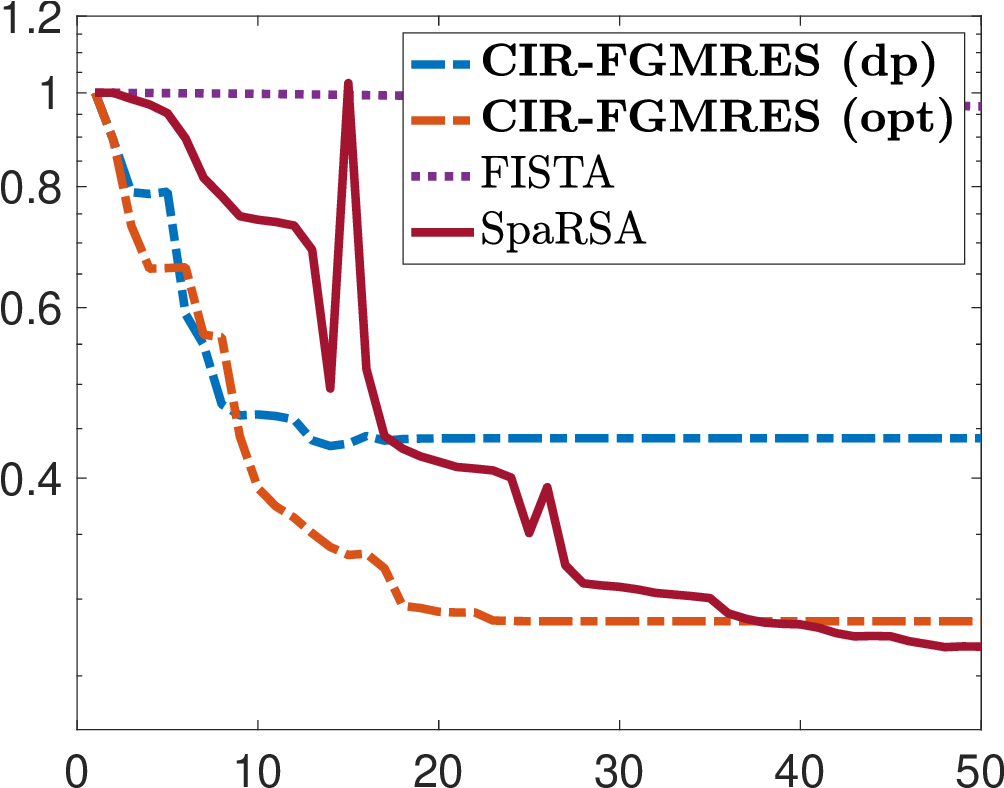}
        \caption*{Other $\ell_1$ solvers}
    \end{subfigure}
    \caption{Different error norm histories for the test problem star cluster.}\label{Ex2_performance}
\end{figure}

Last, in Figure \ref{Ex2_reconstructions}, we show some reconstructions of the solution. As one can expect, the constructions look sparse (except H-GMRES, which does not impose sparsity).

\begin{figure}[ht!]
    \centering
    \begin{subfigure}[b]{0.32\textwidth}
        \centering
        \includegraphics[width=0.95\textwidth]{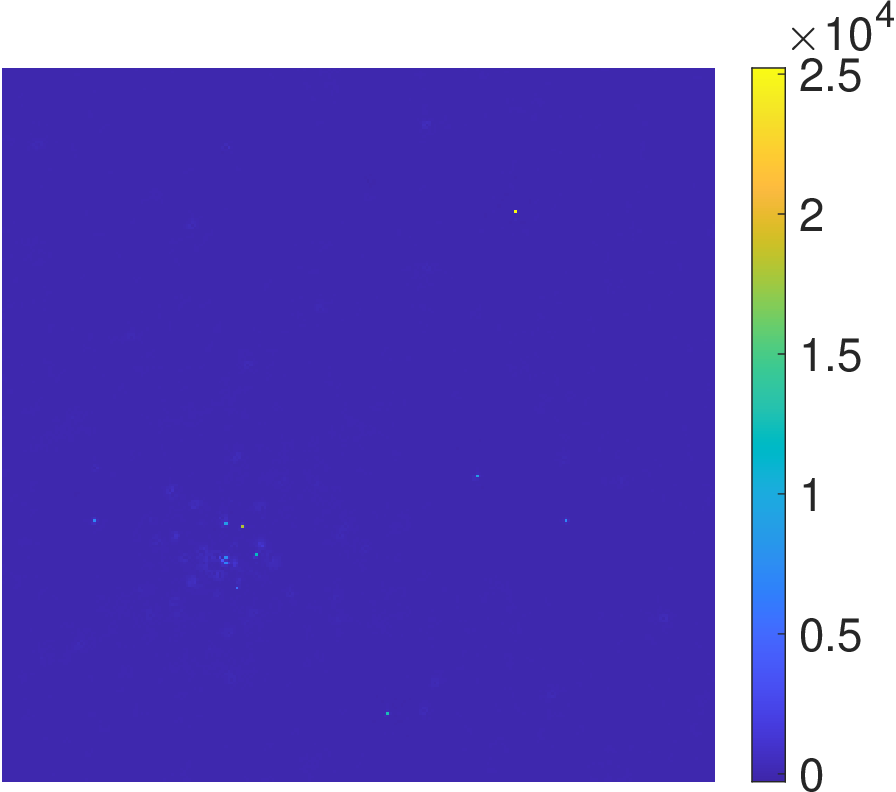}
        \caption*{\bf IR-FGMRES \phantom{\quad \quad}}
    \end{subfigure}
    \begin{subfigure}[b]{0.32\textwidth}
        \centering
        \includegraphics[width=0.95\textwidth]{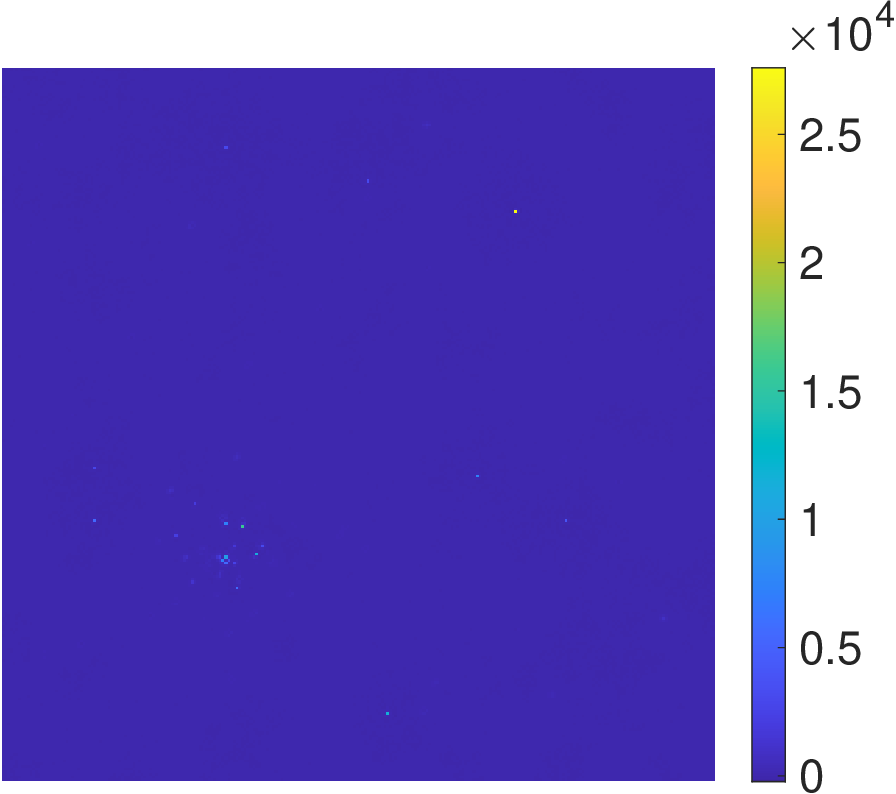}
        \caption*{\bf CIR-FGMRES \phantom{\quad \quad}}
    \end{subfigure}
    \begin{subfigure}[b]{0.32\textwidth}
        \centering
        \includegraphics[width=0.95\textwidth]{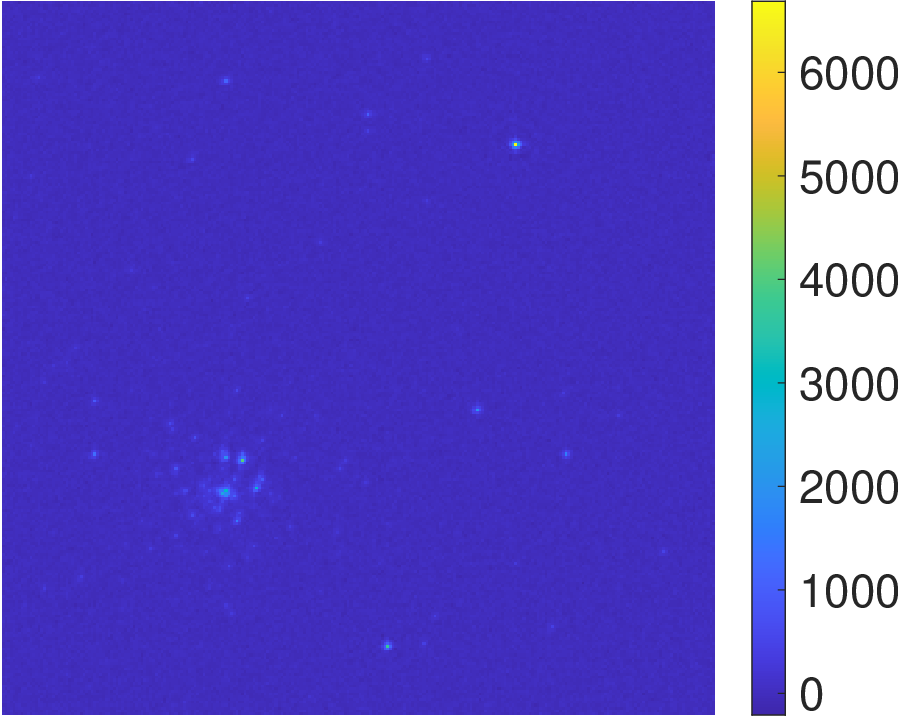}
        \caption*{H-GMRES \phantom{\quad \quad}}
    \end{subfigure}
        \begin{subfigure}[b]{0.32\textwidth}
        \centering
        \includegraphics[width=0.95\textwidth]{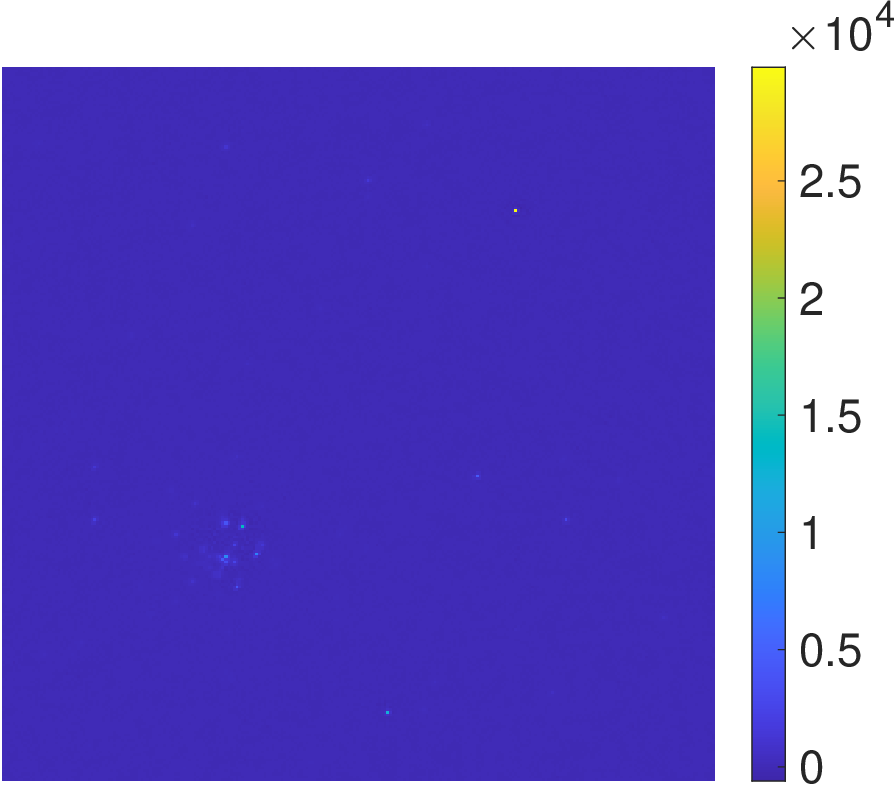}
        \caption*{H-FGMRES \phantom{\quad \quad}}
    \end{subfigure}
    \begin{subfigure}[b]{0.32\textwidth}
        \centering
        \includegraphics[width=0.95\textwidth]{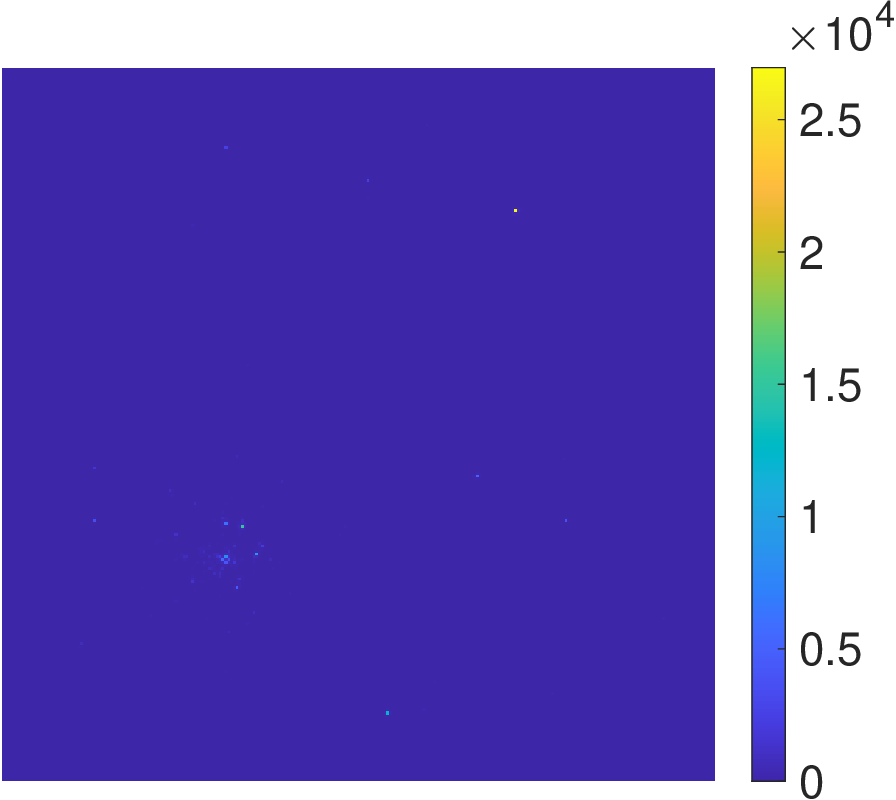}
        \caption*{IRN-GMRES \phantom{\quad \quad}}
    \end{subfigure}
    \begin{subfigure}[b]{0.32\textwidth}
        \centering
        \includegraphics[width=0.95\textwidth]{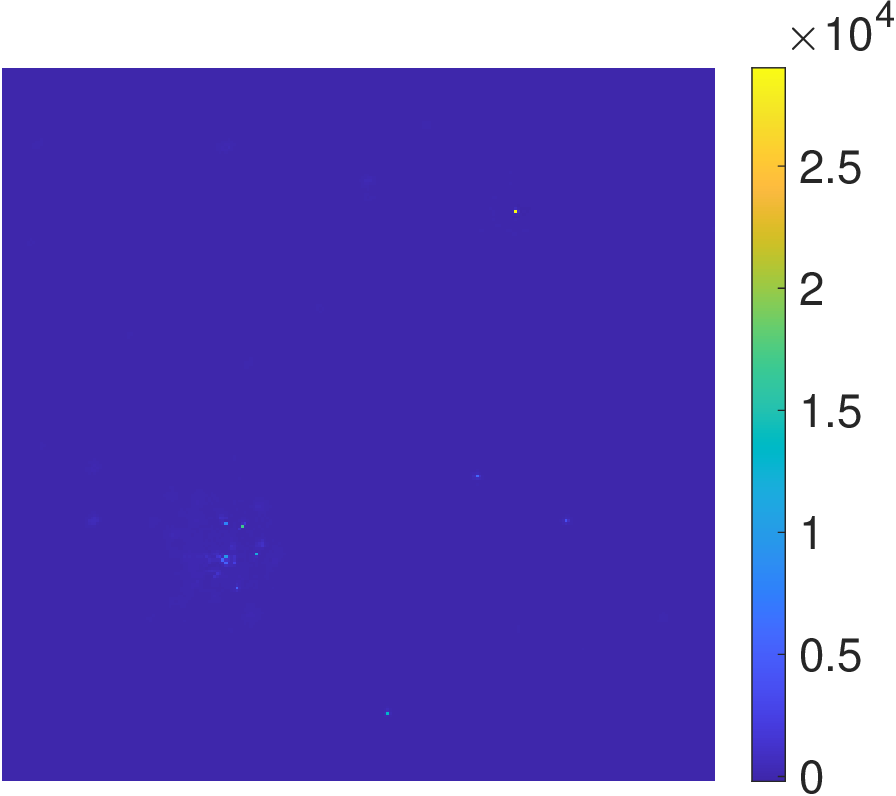}
        \caption*{IRW-FGMRES \phantom{\quad \quad}}
    \end{subfigure}
    \caption{Reconstructions for the test problem star cluster.}
    \label{Ex2_reconstructions}
\end{figure}

\subsection{Oversampled CT problem with high noise level}
The last example corresponds to an oversampled computed tomography (CT) problem with a high level of noise of $50\%$. This type of scenario might occur, for example, when doing fast measurements with low energy to reduce the scanning time and radiation dose, such as the case of chest pathologies and lung cancer screening \cite{Takker2021lowdose,Pinsky2018lowdose}. In this example, we use the Shepp-Logan phantom of size $128\times128$ as the true image and simulate the measuring process $\bfA \in \mathbb{R}^{78192 \times 65536}$ using the function PRtomo in IRtools \cite{IRtools}, modifying the default options so that we measure at 216 equispaced angles between 0 and 179. Both the true image and the noisy sinogram can be observed in Figure \ref{Ex3_setting}.

\begin{figure}[ht!]
    \centering
    \begin{subfigure}[b]{0.47\textwidth}
        \centering
        \includegraphics[width=0.7\textwidth]{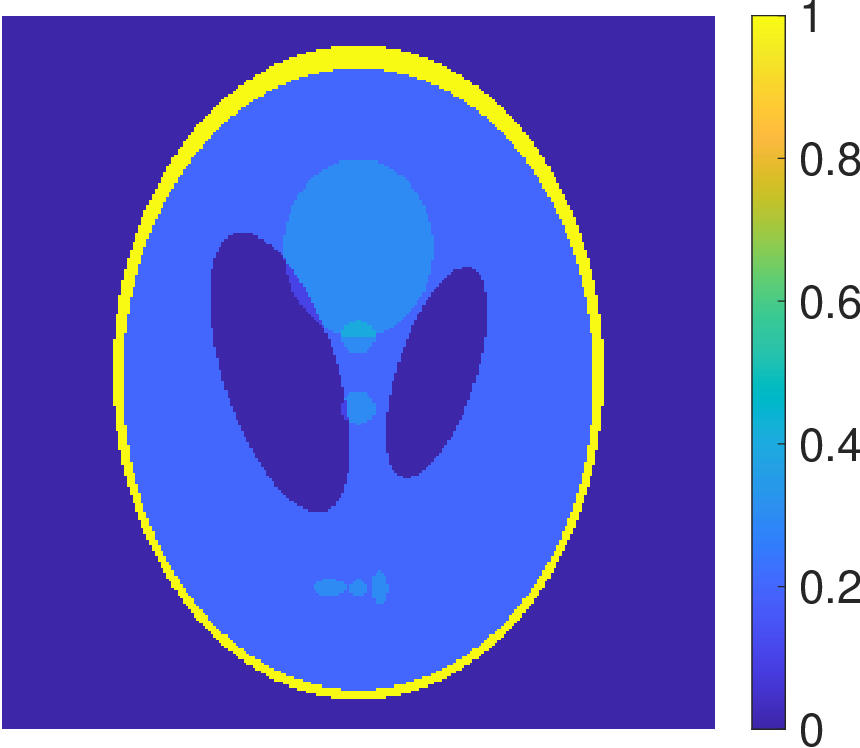}
        \caption*{True solution: $\mathbf{x_{\textbf{true}}}$}
    \end{subfigure}
    \begin{subfigure}[b]{0.47\textwidth}
        \centering
        \includegraphics[width=0.7\textwidth]{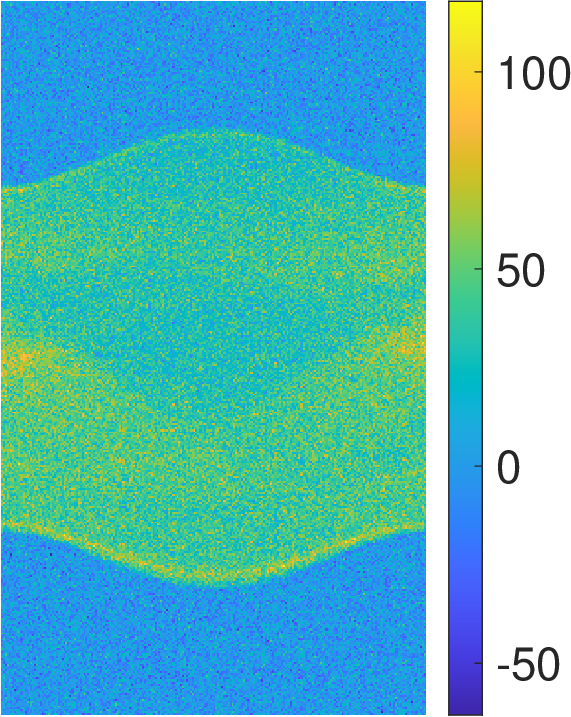}
        \caption*{Noisy measurements: $\mathbf{b}$}
    \end{subfigure}
    \caption{Oversampled highly noisy Shepp-Logan phantom CT problem.}\label{Ex3_setting}
\end{figure}

Similarly to the previous examples, we show different relative error norm histories in Figure \ref{Ex3_error_norms} to showcase the performance of the new methods. Note that the system matrix for this example is not square, so, when comparing Krylov subspace methods, we only compare use those based on the (possibly flexible) Golub-Kahan process. 
In the first panel of Figure \ref{Ex3_error_norms}, one can observe that the discrepancy principle is a very good regularization parameter choice for both new methods: IR-FLSQR and CIR-FLSQR. 

In the second and third panels of Figure \ref{Ex3_error_norms}, one can see the superior performance of the new methods with respect to all the compared methods. Note that all methods that allow restarts (of the solution subspace) use the stabilization of the regularization parameter to trigger the restarts; or when the \revision{number of basis vectors reaches $20$.} \st{number basis reaches 20 vectors.} Therefore, these methods (the new IR-LSQR, CIR-FLSQR and the standard IRN using LSQR as the inner solver) require a limited amount of memory compared to the other Krylov-based methods in the panel 2 of Figure \ref{Ex3_error_norms}. Moreover, recall that all compared methods in the second panel also use the discrepancy principle to automatically select a good regularization parameter at each iteration. However, \revision{the methods in the last panel} required a given regularization parameters as previously described; so we use the one computed by CIR-FLSQR in the last iteration.

\begin{figure}[ht!]
    \centering
    \begin{subfigure}[b]{0.32\textwidth}
        \centering
        \includegraphics[width=0.95\textwidth]{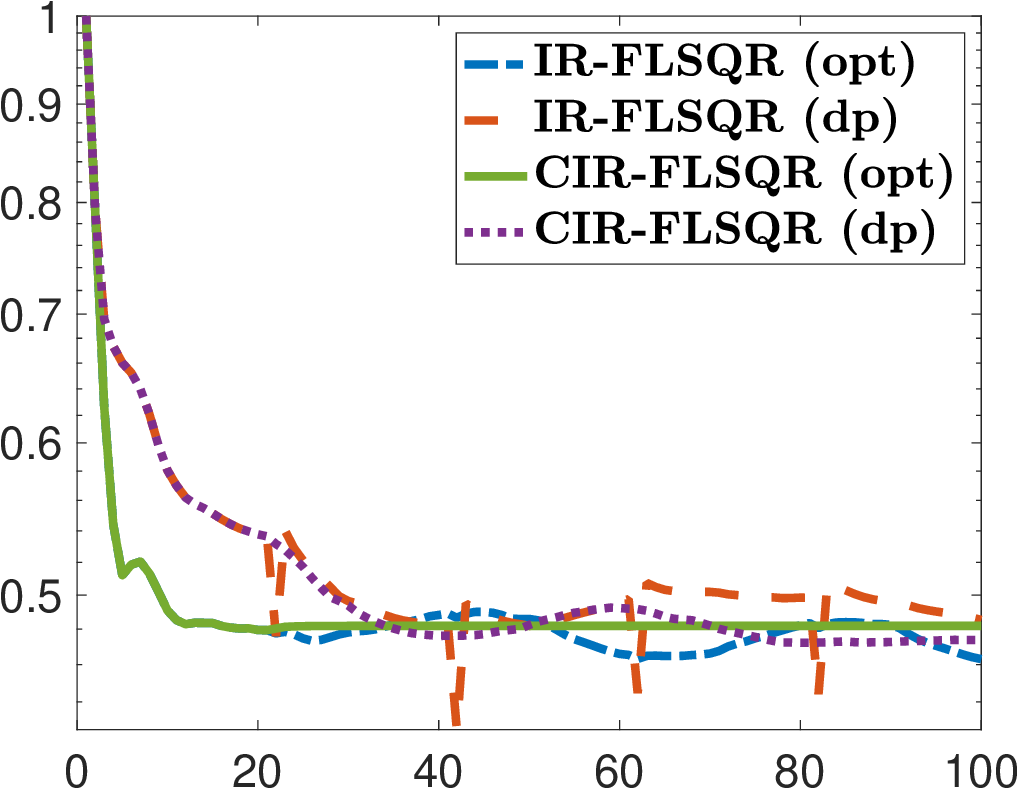}
        \caption*{Different $\lambda_k$ choices}
    \end{subfigure}
    \begin{subfigure}[b]{0.32\textwidth}
        \centering
        \includegraphics[width=0.95\textwidth]{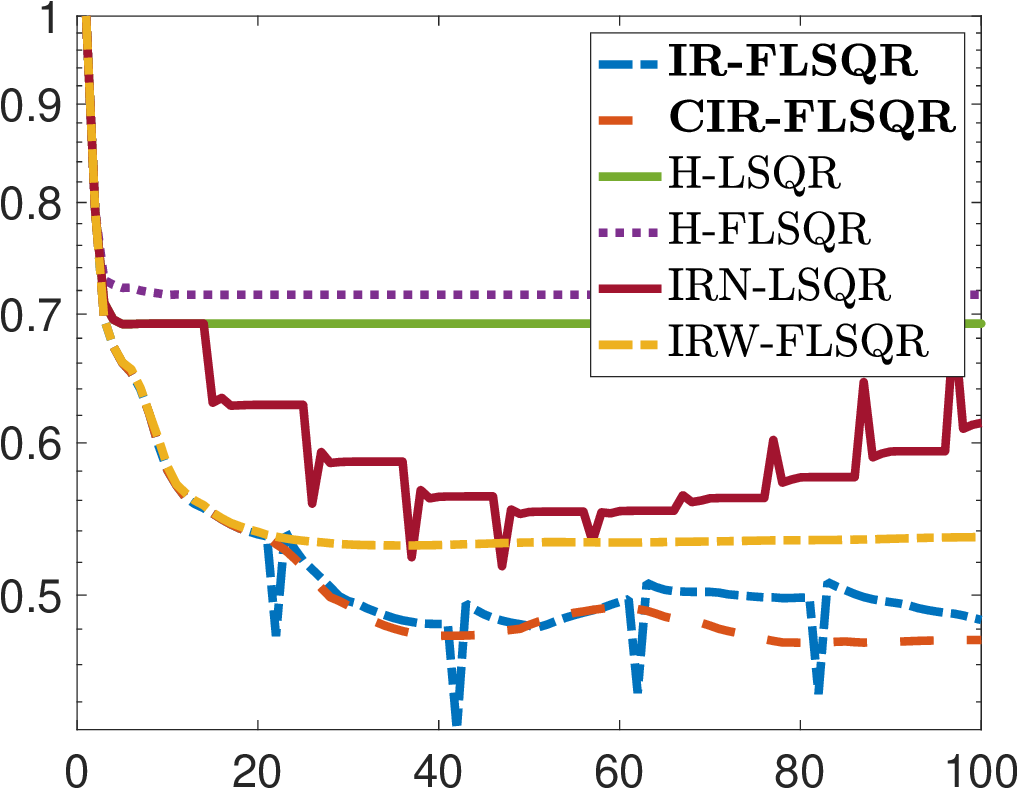}
        \caption*{Other (flexible) Krylov methods}
    \end{subfigure}
    \begin{subfigure}[b]{0.32\textwidth}
        \centering
        \includegraphics[width=0.95\textwidth]{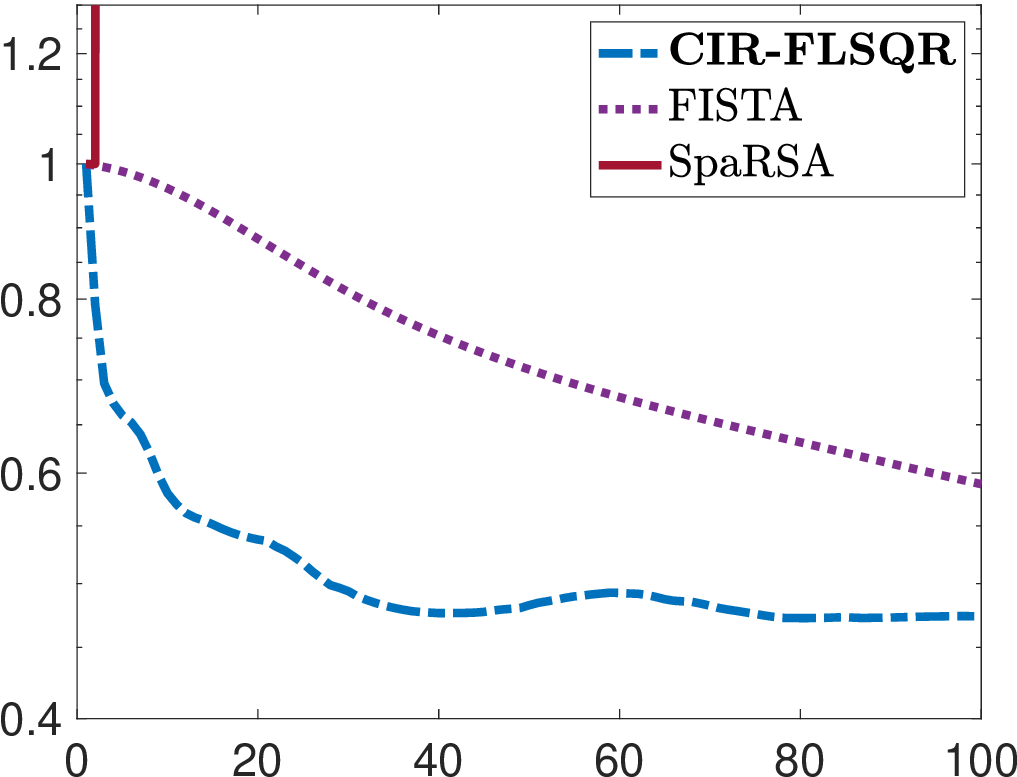}
        \caption*{Other $\ell_1$ solvers}
    \end{subfigure}
    \caption{Different error norm histories for the Shepp-Logan phantom CT problem.}\label{Ex3_error_norms}
\end{figure}

Lastly, in Figure \ref{Ex3_reconstructions}, we can see the reconstructions provided by the new method as well as the other compared (Krylov) solvers. In this figure, one can the excellent performance of these methods; particularly of CIR-FLSQR. Moreover, we recall that the new methods are very efficient in terms of memory by employing suitable restarts, which makes them applicable in real CT setting (unlike methods such as IRW-FLSQR or H-FLSQR).

\begin{figure}[ht!]
    \centering
    \begin{subfigure}[b]{0.32\textwidth}
        \centering
        \includegraphics[width=0.95\textwidth]{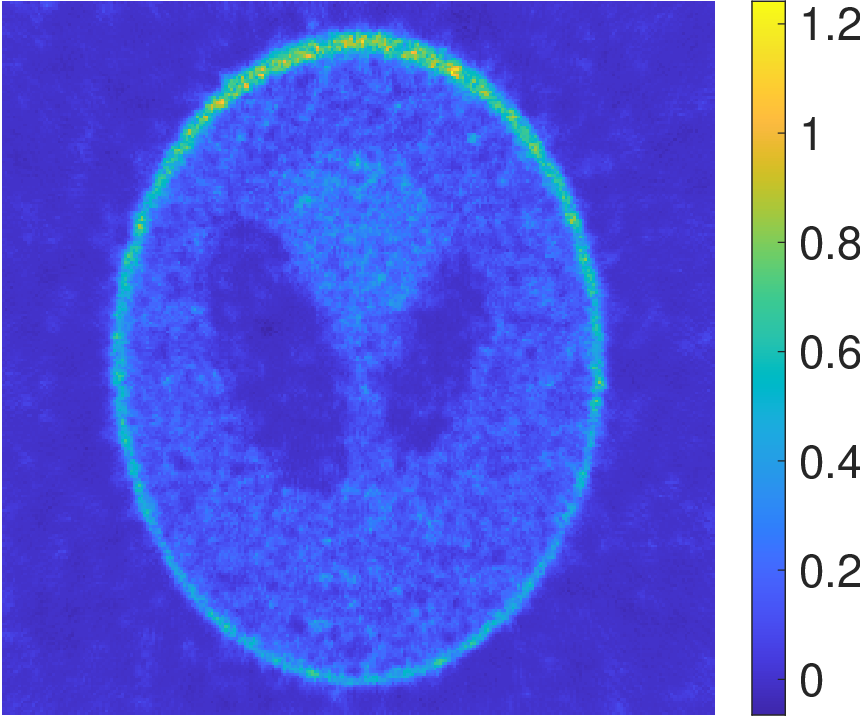}
        \caption*{\bf IR-FLSQR}
    \end{subfigure}
    \begin{subfigure}[b]{0.32\textwidth}
        \centering
        \includegraphics[width=0.95\textwidth]{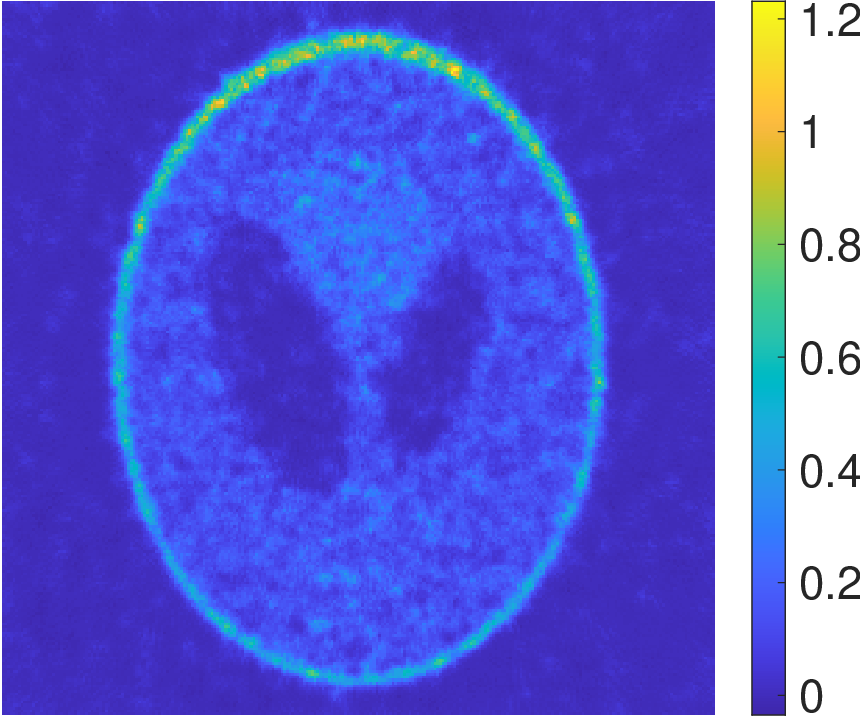}
        \caption*{\bf CIR-FLSQR}
    \end{subfigure}
    \begin{subfigure}[b]{0.32\textwidth}
        \centering
        \includegraphics[width=0.95\textwidth]{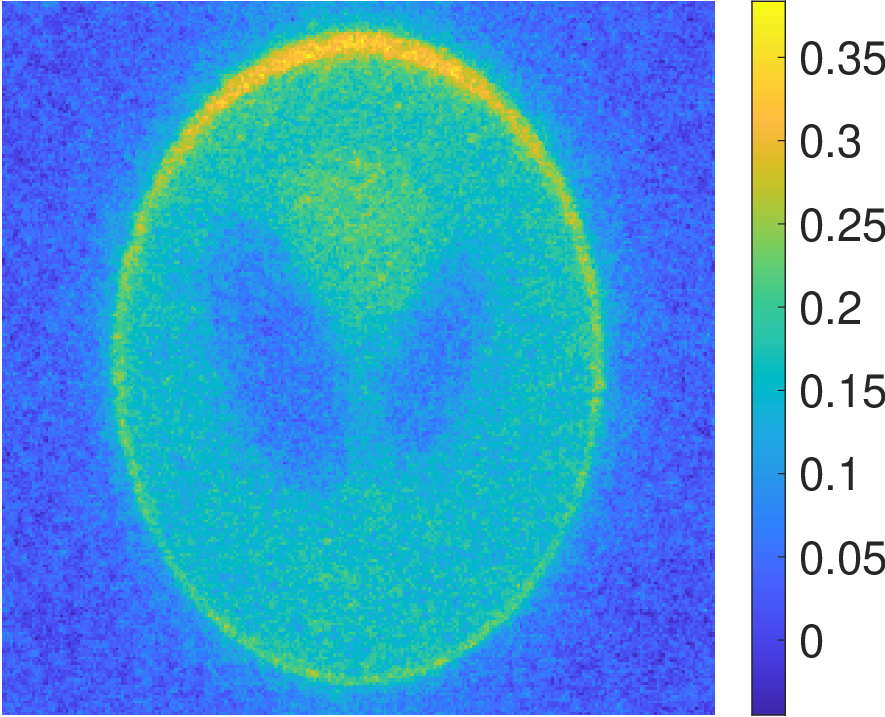}
        \caption*{H-LSQR}
    \end{subfigure}
        \begin{subfigure}[b]{0.32\textwidth}
        \centering
        \includegraphics[width=0.95\textwidth]{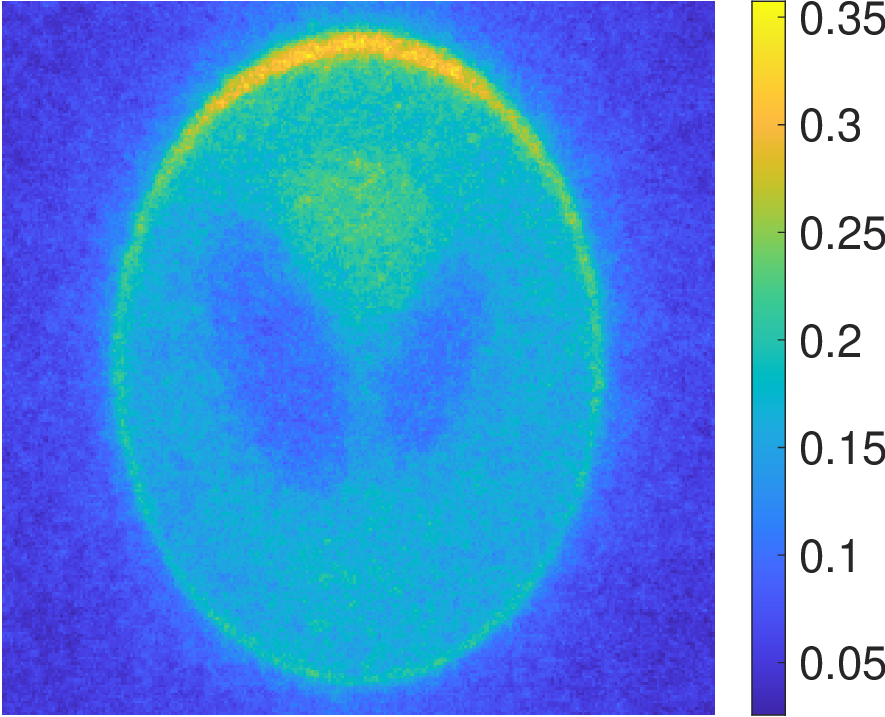}
        \caption*{H-FLSQR}
    \end{subfigure}
    \begin{subfigure}[b]{0.32\textwidth}
        \centering
        \includegraphics[width=0.95\textwidth]{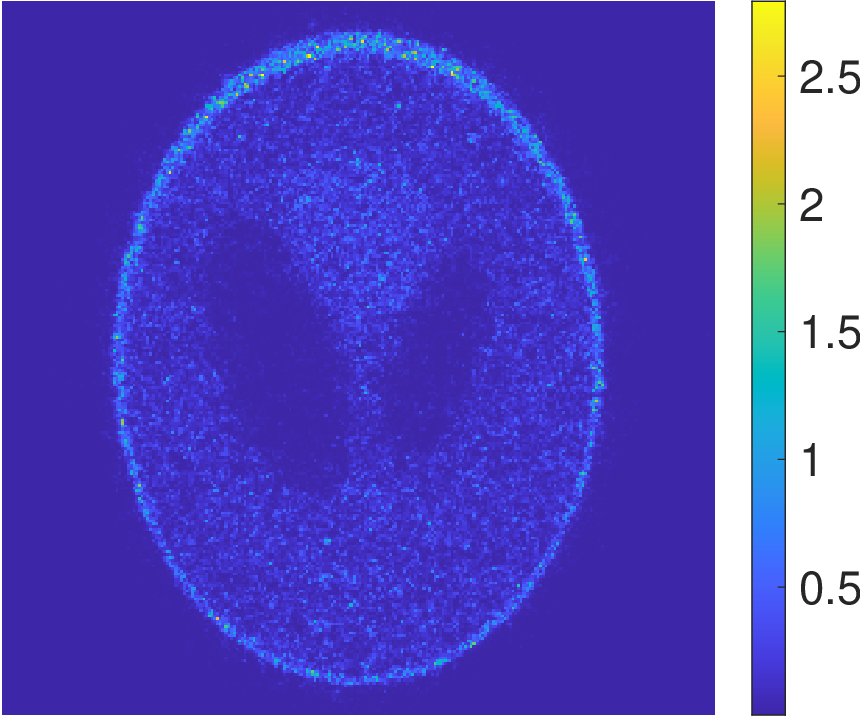}
        \caption*{IRN-LSQR}
    \end{subfigure}
    \begin{subfigure}[b]{0.32\textwidth}
        \centering
        \includegraphics[width=0.95\textwidth]{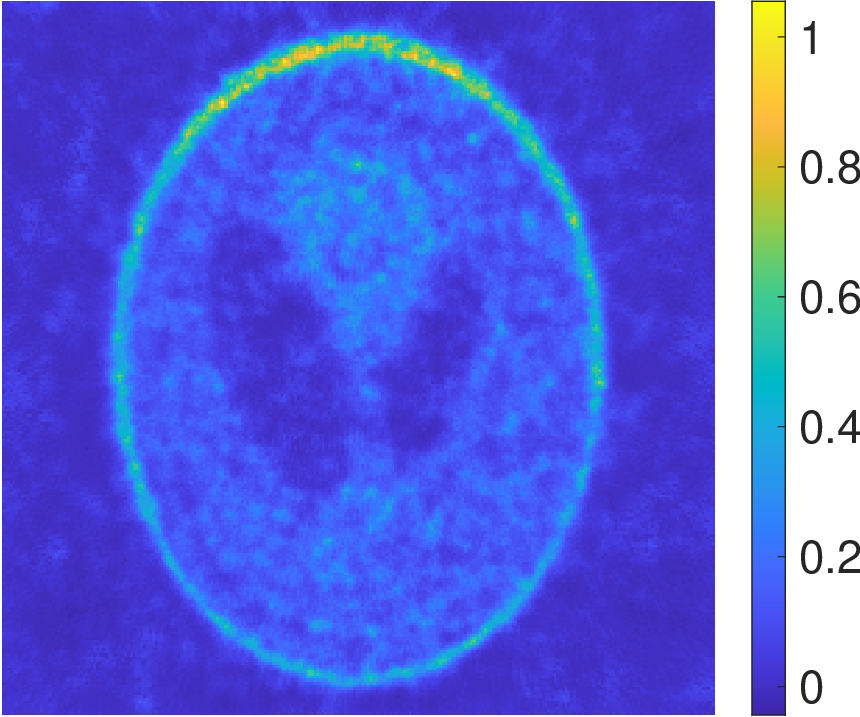}
        \caption*{IRW-FLSQR}
    \end{subfigure}
    \caption{Reconstructions for the highly noisy oversampled CT test problem using the Shepp-Logan phantom.}
    \label{Ex3_reconstructions}
\end{figure}

\section{Conclusion}\label{sec:conclusion}
In this paper, we presented four new algorithms for linear inverse problems with sparse solutions: IR-FGMRES; CIR-FGMRES, IR-FLSQR and CIR-FLSQR. \revision{To do so, we} \st{First, we} proposed a new algorithmic framework for IRN methods, which is based on an iterative refinement interpretation. This setting naturally allows for a restarted scheme \revision{which allows for the automatic selection of the regularization parameters and requires limited memory, making them much more suitable to large-scale problems. We also considered convergence properties of the proposed algorithms and considered several numerical experiments that highlight where our new methods show competitiveness with respect to other standard solvers; particularly in scenarios with high levels of noise.}


\section*{Funding}
LO acknowledges partial support by the U.S. National Science Foundation grants DMS-2038118 and DMS-2208294. MSL acknowledges partial support from the NSF grant DMS-2208294.

\bibliographystyle{alpha}
\bibliography{sample}

\end{document}